\documentclass[12pt,a4paper]{article}
\usepackage[utf8]{inputenc}
\usepackage[english]{babel}
\usepackage{amsmath}
\usepackage{amsfonts}
\usepackage{amssymb}
\usepackage{amsthm}
\usepackage{graphicx}
\usepackage{abstract}
\usepackage{hyperref}
\usepackage[left=3cm,right=3cm,top=2cm,bottom=2cm]{geometry}
\usepackage{accents}
\usepackage{tikz-cd}
\usepackage[nottoc]{tocbibind}

\newcommand{\prC}{\mathbb{C}}
\newcommand{\prQ}{\mathbb{Q}}
\newcommand{\prZ}{\mathbb{Z}}
\newcommand{\prN}{\mathbb{N}}
\newcommand{\prO}{\mathcal{O}}
\newcommand{\prR}{\mathbb{R}}
\newcommand{\prP}{\mathbb{P}}

\newcommand{\prEps}{\varepsilon}
\newcommand{\pre}{\epsilon}
\newcommand{\prm}{\mathfrak{m}}
\newcommand{\praR}{\undertilde{R}}

\newcommand{\prrk}{\operatorname{rk}}
\newcommand{\prbrk}{\operatorname{brk}}
\newcommand{\prsrk}{\operatorname{srk}}
\newcommand{\prbsrk}{\operatorname{bsrk}}

\newcommand{\prid}{\operatorname{id}}

\newcommand{\prHom}{\operatorname{Hom}}

\newcommand{\prGL}{\operatorname{GL}}
\newcommand{\prSL}{\operatorname{SL}}
\newcommand{\prGr}{\operatorname{Gr}}

\newcommand{\prtr}{\operatorname{trace}}

\newcommand{\prSym}{\operatorname{Sym}}
\newcommand{\prsym}{\operatorname{sym}}
\newcommand{\prAnn}{\operatorname{Ann}}
\newcommand{\prProj}{\operatorname{Proj}}

\newcommand{\prExample}{$\diamondsuit$}

\newtheorem{thm}{Theorem}[section]
\newtheorem{cor}[thm]{Corollary}
\newtheorem{prop}[thm]{Proposition}
\newtheorem{lem}[thm]{Lemma}
\newtheorem{conj}[thm]{Conjecture}
\newtheorem{defn}[thm]{Definition}

\theoremstyle{remark}
\newtheorem{remark}[thm]{Remark}
\newtheorem{ex}[thm]{Example}

\newenvironment{numlist}{%

	\begin{enumerate}}%
	{\end{enumerate}}
\newenvironment{abclist}{%

	\begin{enumerate}}%
	{\end{enumerate}}

\title{\textbf{Tensor Rank and Complexity}}
\author{Giorgio Ottaviani\thanks{Universit\`a di Firenze, Italy, \texttt{giorgio.ottaviani@unifi.it} Giorgio Ottaviani is a member of  GNSAGA-INDAM.} \ and
Philipp Reichenbach\thanks{Technische Universit\"at Berlin, Germany, \texttt{reichenbach@tu-berlin.de}. Supported by the European Research Council (ERC) under the European’s Horizon 2020 research and innovation programme (grant agreement no. 787840).}}
\date{}

\begin{document}

%%%% Front Page

\maketitle

\vspace*{-7pt}

\begin{abstract}
These lecture notes are intended as an introduction to several notions of tensor rank and their connections to the asymptotic complexity of matrix multiplication. The latter is studied with the exponent of matrix multiplication, which will be expressed in terms of tensor (border) rank, (border) symmetric rank and the asymptotic rank of certain tensors. We introduce the multilinear rank of a tensor as well, deal with the concept of tensor equivalence and study prehomogeneous vector spaces with the castling transform. Moreover, we treat Apolarity Theory and use it to determine the symmetric rank (Waring rank) of some symmetric tensors.
\end{abstract}

\vspace*{2pt}

\tableofcontents

\newpage

\addcontentsline{toc}{section}{Introduction}
\section*{Introduction}\label{pr:secIntro}

This article grew out of three lectures held by the first author at the Fall School \emph{Varieties, Polyhedra, Computation}, which took place at FU Berlin from September 30 to October 4, 2019 and was part of the Thematic Einstein Semester on Algebraic Geometry in Berlin. The goal of the lectures was to provide an introduction to different notions of tensor rank and their relation to the asymptotic complexity of matrix multiplication. Thereby, special focus has been put on expressing the exponent $\omega$ of matrix multiplication with the various concepts of rank. These extended lecture notes present the covered material in detail, come with additional content and partly aim at offering a survey on the topic.

The first section recalls well-known results on matrix rank and equivalence of matrices. This motivates a natural generalization to $d$-tensors, that is presented in part two. There, the classification of all cases with finitely many tensor equivalence classes is stated. Furthermore, prehomogeneous vector spaces are studied with the help of the castling transform and Venturelli's refinement of the Sato-Kimura Theorem is presented. In the third section the concepts of tensor (border) rank and multilinear rank are introduced. An extended example illustrates how these notions can be used to completely classify tensor equivalence in this particular case.

The fourth part defines the exponent $\omega$ for studying the asymptotic complexity of matrix multiplication. Then, $\omega$ will be phrased in terms of the (border) rank of the tensor of matrix multiplication and two ways of bounding $\omega$ are given. The latter allows to recover Strassen's result from 1969 and his algorithm for multiplying matrices is discussed as well.

Afterwards, the fifth section introduces symmetric rank (also known as Waring rank) and border symmetric rank. To determine the Waring rank of a symmetric tensor part six deals with Apolarity Theory. In particular, the reduced and the scheme-theoretic version of the Apolarity Lemma are presented. This is applied in section seven to several examples, e.g., to determine the symmetric rank of a monomial. In addition, $\omega$ is expressed in terms of the (border) symmetric rank of a family of symmetric tensors.

The last section shows how $\omega$ can be recovered with yet another concept of rank, the asymptotic rank of a tensor. We end with Strassen's Asymptotic Rank Conjecture, its special case for $\omega$ and the introduction of tensor asymptotic rank. Thereafter, the exercises from the fall school are given as an appendix.

The authors do not claim any originality, neither in the presented results, that are all known, nor in the presentation, which is influenced by the vast literature on the subject.

\bigskip

\textbf{Acknowledgements.} The authors wish to thank the organizers of the fall school for the inspiring event and their hospitality, and Peter B\"{u}rgisser for his thorough and helpful revision. The second author would like to thank Miruna-Stefana Sorea for interesting discussions on the exercises.

%%%%%%% Begin: Group Action on Matrices %%%%%%%

\section{Group Action on Matrices}\label{pr:secMatrices}

We gently start into the topic \emph{Tensor Rank and Complexity} via this first section. Namely, we will recall classical knowledge from linear algebra on matrix rank and equivalence of matrices. This shall serve as a main motivation for the corresponding generalizations for tensors, which will be treated in the following Sections~\ref{pr:secTensors} and \ref{pr:secRank}.

Let $V$ and $W$ be finite dimensional $\prC$-vector spaces and set $a_1 := \dim V$, $a_2 := \dim W$. The algebraic group $\prGL(V) \times \prGL(W)$ acts naturally on $V \otimes W$ by
	\begin{align*}
	(g,h) \cdot \left( \sum_{i=1}^r v_i \otimes w_i \right) := \sum_{i=1}^r g(v_i) \otimes h(w_i) \, ,
	\end{align*}
where $(g,h) \in \prGL(V) \times \prGL(W)$, $v_i \in V$ and $w_i \in W$. Here we already used the language of tensors for clear comparison with its generalization in Section~\ref{pr:secTensors}. Still, the action is up to identification just the left-right action on matrices. Since we often identify $V \otimes W$ with the matrix space $\prC^{a_1 \times a_2}$ implicitly, let us explain this once. After fixing bases we can assume $V = \prC^{a_1}$ and $W = \prC^{a_2}$. Interpreting elements of $V$ and $W$ as column vectors we have the natural isomorphism
	\begin{align*}
	V \otimes W \to \prC^{a_1 \times a_2}, \quad \sum_{i=1}^r v_i \otimes w_i \; \mapsto \; \sum_{i=1}^r v_i w_i^T \, ,
	\end{align*}
where $(\cdot)^T$ denotes the transposition. Note that non-zero vectors of the form $v \otimes w$ are bijectively identified with matrices of rank one. Furthermore, the $\prGL(V) \times \prGL(W)$ action becomes under this identification
	\begin{align*}
	\big( \prGL_{a_1}(\prC) \times \prGL_{a_2}(\prC) \big) \times \prC^{a_1 \times a_2} \to \prC^{a_1 \times a_2}, \quad (g,h,M) \mapsto gMh^T \, ,
	\end{align*}
i.e., the left-right action on $\prC^{a_1 \times a_2}$. Remember that two matrices are called \emph{equivalent}, if they lie in the same orbit under the left-right action. For matrix equivalence there is the following theorem.

\begin{thm}
The natural $\prGL(V) \times \prGL(W)$ action on $V \otimes W$ has finitely many orbits, which are parametrized by the matrix rank.
\end{thm}

\begin{proof}
For $(g,h) \in \prGL(V) \times \prGL(W)$ and a matrix $M \in V \otimes W$ it holds that $\prrk(M) = \prrk(gMh^T)$. Thus, any orbit only contains matrices of the same rank. In fact, any matrix $M$ can be transformed by Gaussian elimination from left and right into its \emph{rank normal form}
	\begin{align*}
	\begin{bmatrix}
	I_r & 0 \\ 0 & 0
	\end{bmatrix} \in \prC^{a_1 \times a_2} \cong V \otimes W \, ,
	\end{align*}
where $0 \leq r \leq \min \lbrace a_1, a_2 \rbrace$ is the rank of $M$ and $I_r \in \prC^{r \times r}$ denotes the identity matrix. Hence, for each possible rank there is exactly one orbit.
\end{proof}

From the algebraic geometry perspective this action behaves nicely. Indeed, the Zariski-closure of the orbit for rank $r$ is given by
	\begin{align*}
	X_r := \lbrace M \in V \otimes W \mid \prrk(M) \leq r \rbrace
	\end{align*}
and this algebraic variety is well-understood. It is the vanishing locus of all $(r+1) \times (r+1)$ minors and the singular locus of $X_r$ is given by $X_{r-1}$. Therefore, the orbits $X_r \setminus X_{r-1}$ of the action are all smooth. Moreover, $X_r$ is even a cone. In particular, $X_1$ is the affine cone of the Segre variety $\prP(V) \times \prP(W) \subseteq \prP(V \otimes W)$, because a vector $v \otimes w$ is either zero or corresponds to a matrix of rank one.

In addition, the orbit closures $X_r$ yield an ascending chain of inclusions
	\begin{equation}\label{pr:eqXrMatrixCase}
	\lbrace 0 \rbrace = X_0 \subseteq X_1 \subseteq X_2 \subseteq \ldots \subseteq X_{\min\lbrace a_1, a_2 \rbrace} = V \otimes W \, .
	\end{equation}%TODO picture for orbit closure? or just this equation?
The last equality expresses the fact that the matrices of maximal rank are Zariski-dense in $V \otimes W$.

\begin{lem}\label{pr:lemSec1} Let $0 \leq r,s \leq \min \lbrace a_1, a_2 \rbrace$.
	\begin{abclist}
	\item If $A \in X_r$ and $B \in X_s$, then $A + B \in X_{r+s}$.
	\item It holds that $X_r = \big\lbrace \sum_{i=1}^r A_i \mid  A_i \in X_1 \big\rbrace$.
	\end{abclist}
\end{lem}

\begin{proof}
Part~a) is a consequence of $\prrk(M_1 + M_2) \leq \prrk(M_1) + \prrk(M_2)$ for matrices $M_i$. For part~b) the inclusion ``$\supseteq$'' is immediate from part~a). Conversely, for any matrix $M \in X_r$ of rank $s \leq r$ there is $(g,h) \in \prGL_{a_1}(\prC) \times \prGL_{a_2}(\prC)$ such that $gMh^T$ is in rank normal form. Hence
	\begin{align*}
	M = g \begin{bmatrix} I_{s} & 0 \\ 0 & 0 \end{bmatrix} h^T
	= \sum_{i=1}^s \underbrace{g_i h_i^T}_{\quad \in X_1} \, ,
	\end{align*}
where $g_i \in \prC^{a_1}$ (respectively $h_i \in \prC^{a_2}$) is the $i$-th column of $g$ (respectively $h$). If $s < r$ we may fill up the sum with zeros, using $0 \in X_1$.
\end{proof}

In particular, any $X_r$ is determined by $X_1$ using part b) of Lemma~\ref{pr:lemSec1}. This will be used as a motivation for the definition of \emph{tensor rank} in Section~\ref{pr:secRank}.

%%%%%%% End: Group Action on Matrices %%%%%%%

%%%%%%% Begin: Group Action on Tensors %%%%%%%

\section{Group Action on Tensors}\label{pr:secTensors}

Now, we generalize the setting from Section~\ref{pr:secMatrices} to tensor products with $d$ factors. For this, let $V_1, \ldots, V_d$ be finite dimensional $\prC$-vector spaces and set $a_i := \dim V_i$ for $i =1,\ldots,d$. The algebraic group $G := \prGL(V_1) \times \cdots \times \prGL(V_d)$ acts naturally on $V:= V_1 \otimes \cdots \otimes V_d$ via the linear maps $g_1 \otimes \cdots \otimes g_d$, i.e.,
\begin{align*}
(g_1,\ldots,g_d) \cdot \left( \sum_{j=1}^r v_{1,j} \otimes \cdots \otimes v_{d,j} \right) := \sum_{j=1}^r \, g_1 (v_{1,j}) \otimes \cdots \otimes g_d (v_{d,j}) \, ,
\end{align*}
where $v_{i,j} \in V_i$ for $i =1, \ldots, d$. Of course, we may always assume $V_i = \prC^{a_i}$ after fixing a basis on $V_i$.

Given this setting, one is interested in understanding the action of $G$ on $V$ in the new cases $d \geq 3$. In analogy to the matrix case we may call two tensors in $V$ \emph{equivalent}, if they lie in the same $G$-orbit. Unfortunately, the $G$-action, and thus the notion of tensor equivalence, will turn out to be much less well-behaved than in the matrix case $d=2$. To guide ourselves, we ask the following questions for $d \geq 3$:
	\begin{numlist}
	\item Has the natural action of $G$ on $V$ finitely many orbits?
	\item Is there a $G$-orbit, which is Zariski-dense in $V$?
	\item Can one classify the orbits of the natural $G$-action on $V$?
	\end{numlist}
Before we investigate these questions, we give a definition regarding the second question. It is motivated by calling a variety $X$, which is equipped with a transitive action of an algebraic group $H$, a homogeneous space.

\begin{defn}
$V = V_1 \otimes \cdots \otimes V_d$ is called \textbf{prehomogeneous} if the natural action of $G$ on $V$ has a Zariski-dense orbit. Hence, the second question asks if $V$ is prehomogeneous.
\end{defn}

Now, let us consider the three questions from above. Starting with the bad news, in general the answer to the third question is ``No, this is hopeless''; and there is a mathematical reasoning for that! Namely, in the language of representation theory there are the so-called \textit{wild problems}. These wild problems refrain themselves from classification, since they contain for \emph{all}(!) $m \geq 2$ the problem of classifying $m$-tuples of matrices up to simultaneous similarity. In our situation, \cite[Theorem~4.5]{prBSwildProblem} states that already the classification of the $G$-orbits for $\prC^3 \otimes \prC^m \otimes \prC^n$ is a wild problem. Using this, one can deduce that most instances of classifying $d$-tensors for $d \geq 3$ are wild. An interested reader may consult the article \cite{prBSwildProblem} and the references therein for further details.

Although the classification problem is in general out of reach, there are complete classifications for certain $V$. The easiest case $\prC^2 \otimes \prC^2 \otimes \prC^2$ is presented in Section~\ref{pr:secRank}. Moreover, one can always decide the first two questions and we dedicate the subsequent part of this section to these two questions. A full answer to question~1 respectively question~2 will be given in Theorem~\ref{pr:thmFiniteOrbits} respectively Theorem~\ref{pr:thmSatoKimuraVenturelli} below.

We start the study with some considerations, that are useful for both question~1 and question~2. Note that if $V$ has finitely many $G$-orbits, then there is always a Zariski-dense orbit, i.e., $V$ is prehomogeneous. This is due to the stratification of $V$ by the $G$-orbits and due to $V$ being irreducible. Conversely, we will see that not all prehomogeneous $V$ have finitely many orbits. Hence, question~1 is more restrictive.

\begin{remark}
The dimension formula of an orbit $G \cdot v$ ($v \in V$) provides a useful necessary condition for prehomogeneity, and hence in particular for having finitely many orbits.

To see this, let $V = \prC^{a_1} \otimes \cdots \otimes \prC^{a_d}$ and choose $v \in V$. The stabilizer $G_v$ of $v$ contains by definition of the $G$-action the $d-1$ dimensional torus
	\begin{align*}
	T:= \big\lbrace (\lambda_1 I_{a_1}, \ldots, \lambda_d I_{a_d}) \mid \lambda_i \in \prC^\times \, , \; \lambda_1 \cdots \lambda_d = 1 \big\rbrace \subseteq G \, ,
	\end{align*}
where $I_m \in \prC^{m \times m}$ denotes the identity matrix. Thus, the dimension formula for the orbit $G \cdot v$ yields
	\begin{align*}
	\dim G \cdot v = \dim G - \dim G_v \leq \dim G - \dim T = \sum_{i=1}^d a_i^2 - d + 1 \, .
	\end{align*}
Comparing this with the dimension of $V$, we see that whenever
	\begin{align*}
	N(a_1, \ldots, a_d) := (\dim G - \dim T) - \dim V = 1-d + \sum_{i=1}^d a_i^2 - \prod_{i=1}^d a_i < 0 \, ,
	\end{align*}
$V$ cannot be prehomogeneous as $\dim G \cdot v < \dim V$ for any $v \in V$. In particular, $V$ has infinitely many orbits if $N(a_1,\ldots,a_d) < 0$.

On the other hand, $N(a_1,\ldots,a_d) \geq 0$ does not necessarily imply prehomogeneity as we will see below, compare Theorem~\ref{pr:Thm2kk}.
\end{remark}

\begin{ex}\label{pr:exNcomputation}
Direct computation for $n \in \prN$ gives
	\begin{numlist}
	\item $N(2,2,n) = (n-2)^2 + 2$
	\item $N(2,3,n) = (n-3)^2 + 2$
	\item $N(2,n,n+1) = 3$
	\item $N(2,n,n) = 2$
	\item $N(n,\ldots,n) = 1-d + dn^2 - n^d$, in particular, $V = (\prC^n)^{\otimes d}$ is not prehomogenous if $n$ and $d$ are at least three.
	\end{numlist}
\end{ex}

Let us now state the solution for the first question. Of course, if
	\begin{align*}
	\alpha := \big\vert \lbrace i \mid a_i > 1 \rbrace   \big\vert \leq 2
	\end{align*}
then, as a consequence of the matrix case from Section~\ref{pr:secMatrices}, $V$ has finitely many orbits. Besides this, only the first two types of tuples from Example~\ref{pr:exNcomputation} admit finitely many orbits:

\begin{thm}[\cite{prKacNilpotentOrbits}]\label{pr:thmFiniteOrbits}
Let $d \geq 3$ and $2 \leq a_1 \leq \ldots \leq a_d$. The natural action of $G$ on $V$ has only in the following cases finitely many orbits:
	\begin{numlist}
	\item $d=3$ and $(a_1,a_2,a_3) = (2,2,n)$ for some $n \geq 2$.
	\item $d=3$ and $(a_1,a_2,a_3) = (2,3,n)$ for some $n \geq 3$.
	\end{numlist}
\end{thm}

\begin{proof}
For a proof we refer to \cite[Proposition~30]{prManivelLectures}.
\end{proof}

Although $V$ has finitely many orbits in the mentioned cases, the orbit structure is differently behaved compared to the matrix case. This is already witnessed by the simplest possible case $\prC^2 \otimes \prC^2 \otimes \prC^2$.

\begin{ex}
For $V = \prC^2 \otimes \prC^2 \otimes \prC^2$ there are seven orbits (including the zero orbit) and the containment graph for their orbit closures is given in Figure~\ref{pr:figAffineOrbitsC2}.  (See Section~\ref{pr:secDim2} for the proof.) This containment graph differs from the matrix case. Indeed, it is not just one ascending chain of inclusions as in equation~\eqref{pr:eqXrMatrixCase}.
\begin{figure}[ht]
  \centering
  \begin{tikzcd}[column sep = small]
  & \,V \ar[d, dash] &  \\ 
  & \bullet \ar[d, dash] \ar[ld, dash] \ar[rd, dash] & \\ 
  \bullet \ar[rd, dash] & \bullet \ar[d, dash] & \bullet \ar[ld, dash] \\
  & \bullet \ar[d, dash] & \\
  & \lbrace 0 \rbrace &
  \end{tikzcd}
  \caption{Containment Graph for orbit closures of $V = \prC^2 \otimes \prC^2 \otimes \prC^2$}
  \label{pr:figAffineOrbitsC2}
\end{figure}
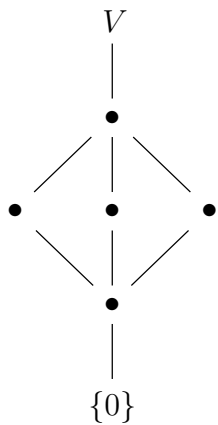

Moreover, we will see in Section~\ref{pr:secRank} that the generalization of matrix rank for tensors is not enough to distinguish between the orbits of $\prC^2 \otimes \prC^2 \otimes \prC^2$, compare Table~\ref{pr:tabRepresentativesC2}. For further illustrations of Theorem~\ref{pr:thmFiniteOrbits} we refer to \cite{prParfenovFiniteNumber} and \cite[Section~10.3]{prLandsbergBook}.
\end{ex}

As a consequence of Theorem~\ref{pr:thmFiniteOrbits} the tuples $(2,n,n)$ and $(2,n,n+1)$ do not give rise to $V$ with finitely many orbits if $n \geq 4$. Still, $\prC^2 \otimes \prC^n \otimes \prC^{n+1}$ is always prehomogeneous by the next theorem. Thus, $(2,n,n+1)$ for $n \geq 4$ provides a full family of prehomogeneous $V$ with \emph{infinitely} many orbits.

\begin{thm}
For $n \in \prN$ the tensor product $\prC^2 \otimes \prC^n \otimes \prC^{n+1}$ is prehomogeneous.
\end{thm}

\begin{proof}
It follows from \cite[Proposition~9.4]{prHarris}.
%Since $N(2,n,n+1) = 3$ the statement is a special case of \cite[Theorem~2.5]{prVenturelli}.
\end{proof}

Also the following result for $(2,n,n)$ is of interest. Namely, it gives a family of examples showing that $N(a_1,\ldots,a_d) \geq 0$ may not be enough for being prehomogeneous. (Recall $N(2,n,n)=2$ from Example~\ref{pr:exNcomputation}.)

\begin{thm}\label{pr:Thm2kk}
$\prC^2 \otimes \prC^n \otimes \prC^{n}$ is prehomogeneous if and only if $n \leq 3$.
\end{thm}

\begin{proof}
We refer to \cite[Theorem~23]{prManivelLectures}. 
\end{proof}

Let us turn to the general case $\prC^{a_1} \otimes \cdots \otimes \prC^{a_d}$.

\begin{remark}
If $a_1 \geq \prod_{i=2}^d a_i$, then $\prC^{a_1} \otimes \cdots \otimes \prC^{a_d}$ is prehomogeneous.
\end{remark}

\begin{proof}
Set $V := \prC^{a_1}$, $W := \prC^{a_2} \otimes \cdots \otimes \prC^{a_d}$ and $H := \prGL_{a_2}(\prC) \times \cdots \times \prGL_{a_d}(\prC)$. In the following we view $V \otimes W$ as a matrix space. Note that the action of $\prGL_{a_1}(\prC) \times H$ on $V \otimes W$ is in general \emph{not} the action from Section~\ref{pr:secMatrices}, because $H$ is only a subgroup of $\prGL(W)$. Nevertheless, the matrices of full rank $m := a_2 \cdots a_d$ form a Zariski-dense set in $V \otimes W \cong \prC^{a_1 \times m}$, since $a_1 \geq m$. Moreover, any matrix of full rank can be transformed by left multiplication with some $g \in \prGL_{a_1}(\prC)$ (Gau{\ss} algorithm) to
	\begin{align*}
%	\begin{bmatrix}
%	1 &  &  \\ 
%	 & \ddots &  \\ 
%	 & & \ddots & \\
%	 &  & & 1 \\ 
%	0 & \cdots & \cdots & 0 \\ 
%	\vdots & & & \vdots \\ 
%	0 & \cdots & \cdots & 0
%	\end{bmatrix} \in \prC^{a_1 \times m} \, .
	\begin{bmatrix} I_m \\ 0 \end{bmatrix} \in \prC^{a_1 \times m} \, ,
	\end{align*}
where $I_m$ is the identity matrix. Thus, the full rank matrices are contained in an orbit of the $\prGL_{a_1}(\prC) \times H$ action. (The $H$ action is not needed for the latter.) Actually, the full rank matrices form a $\prGL_{a_1}(\prC) \times H$ orbit, because multiplying with elements of $H$ from the right preserves the rank. Hence, $V \otimes W$ has a dense orbit, i.e., is prehomogeneous.
\end{proof}

Therefore, the interesting case for studying prehomogeneity is when $a_i < \prod_{i \neq j}^d a_j$ holds for all $i =1,\ldots,d$. In this case the following proposition (e.g., with $k = a_d$ and $n = a_1 \cdots a_{d-1}$) is very useful and will lead us to the notion of castling transforms.

\begin{prop}\label{pr:propManivel}
Fix $1 \leq k < n$ and let $V$ be an $n$-dimensional $\prC$-vector space.  Consider a rational representation $\varrho \colon G \to \prGL(V)$ of an algebraic group $G$ and its dual representation $\varrho^{*} \colon G \to \prGL(V^{*})$.  Denote by $(\prGr(k,V), \varrho)$ the induced action of $G$ on the Grassmannian $\prGr(k,V)$ of $k$-planes in $V$. Then the following holds.
	\begin{numlist}
		\item $(\prGr(k,V), \varrho)$ has a Zariski-dense $G$-orbit if and only if $(V \otimes \prC^k, \varrho \otimes \prid_{k})$ has a Zariski-dense $(G \times \prGL_k(\prC))$-orbit.
		
		\item $(V \otimes \prC^k, \varrho \otimes \prid_{k})$ has a Zariski-dense $(G \times \prGL_k(\prC))$-orbit if and only if $(V^* \otimes \prC^{n-k}, \varrho^* \otimes \prid_{n-k})$ has a Zariski-dense $(G \times \prGL_{n-k}(\prC))$-orbit.
	\end{numlist}
\end{prop}

\begin{proof}[Sketch of proof.]
This can be found in \cite[Proposition~28 and proof]{prManivelLectures}  as well as in \cite[§2 Proposition~7 and proof]{prSatoKimura}.  To motivate the upcoming definition of castling transforms we sketch the main ideas of the proof.

To prove item 1,  let $e_1, \ldots,e_k$ be the standard basis of $\prC^k$. The group $G \times \prGL_k(\prC)$ acts on $V \otimes \prC^k \cong V^k$ by $(g,h) \mapsto \varrho(g) \otimes h$ and leaves the Zariski-open and Zariski-dense subset
	\begin{align*}
	U := \left\lbrace \sum_{i=1}^k v_i \otimes e_i \mid v_1,\ldots,v_k \text{ linearly independent}  \right\rbrace
	\end{align*}
invariant. 
We view $\prGr(k,V)$ as a closed subvariety of $\prP \big( \bigwedge^k V \big)$ via the Pl\"ucker embedding and consider
	\begin{align*}
	f \colon U \to \prGr(k,V), \quad \sum_{i=1}^k v_i \otimes e_i \mapsto [v_1 \wedge \ldots \wedge v_k].
	\end{align*}
The map $f$ is algebraic, surjective and equivariant in the sense that $f \big( (g,h)\cdot u \big) = g \cdot f(u)$ for all $(g,h) \in G \times \prGL_k(\prC)$ and all $u \in U$. Moreover, the fibre of a point $W \in \prGr(k,V)$ is given by all bases of $W$ and hence $\prGL_k(\prC)$ acts transitively on each fibre; this action exactly corresponds to a change of bases.

Now, we deduce the first statement from the properties of $f$.  In the following, we always refer to the Zariski topology. First, recall that an orbit under an algebraic group is open in its closure. Hence, a dense orbit is open in the ambient space. That said, if $X$ is a dense $G$-orbit in $\prGr(k,V)$, then $X$ is open in $\prGr(k,V)$ and hence $f^{-1}(X)$ is open in $U$ by continuity.  Since $U$ is open, $f^{-1}(X)$ is open in $V \otimes \prC^k$ as well, and the surjectivity of $f$ yields $f^{-1}(X) \neq \emptyset$. By irreducibility of $V \otimes \prC^k$, $f^{-1}(X)$ is dense in $V \otimes \prC^k$.  Actually, $f^{-1}(X)$ is a $G \times \prGL_k(\prC)$-orbit by the property of the fibres of $f$. This shows one direction of item~1.

On the other hand, if $Y$ is a dense $G \times \prGL_k(\prC)$-orbit in $V \otimes \prC^k$, then $Y$ is open. Thus, the irreducibility of $V \otimes \prC^k$ implies $Y \cap U \neq \emptyset$. Choosing some $y \in Y \cap U$, we have $\big( G \times \prGL_k(\prC) \big) \cdot y = Y \subseteq U$, because $U$ is $G \times \prGL_k(\prC)$-stable. By equivariance of $f$, $f(Y)$ is a  $G$-orbit in $\prGr(k,V)$. If $\overline{Y}^U$ denotes the closure of $Y$ in $U$, we have
	\begin{align*}
	\prGr(k,V) = f(U) = f \left( \overline{Y}^U \right) \subseteq \overline{f(Y)} \subseteq \prGr(k,V) \, ,
	\end{align*}
where we used surjectivity of $f$ in the first equality and continuity of $f$ in the first containment. We have shown $\overline{f(Y)} = \prGr(k,V)$, which ends the proof of item~1. 

Finally, the actions $(\prGr(k,V), \varrho)$ and $(\prGr(n-k,V^*), \varrho^*)$ are naturally isomorphic. Hence, applying item~1 to these two actions yields the second statement.
\end{proof}

\begin{thm}\label{pr:thmCastling}
Let $(a_1,\ldots,a_d) \in \prN^d$ be such that $b_d := \prod_{i=1}^{d-1} a_i - a_d > 0$.  Then $\prC^{a_1} \otimes \cdots \otimes \prC^{a_d}$ is prehomogeneous if and only if $\prC^{a_1} \otimes \cdots \otimes \prC^{a_{d-1}} \otimes \prC^{b_d}$ is prehomogeneous.
\end{thm}

\begin{proof}
Set $n:= \prod_{i=1}^{d-1} a_i$ and $k := a_d$. Then $n > k$ by assumption. Moreover, for $G := \prGL_{a_1}(\prC) \times \cdots \times \prGL_{a_{d-1}}(\prC)$ and $V := \prC^{a_1} \otimes \cdots \otimes \prC^{a_{d-1}}$ let $\varrho$ be the natural representation of $G$ on $V$.  By Proposition~\ref{pr:propManivel}, $(V \otimes \prC^k, \varrho \otimes \prid_{k})$ is prehomogeneous if and only if $(V^* \otimes \prC^{n-k}, \varrho^* \otimes \prid_{n-k})$ is prehomogeneous. On the other hand, given the automorphism
	\begin{align*}
	\varphi \colon G \to G, \quad (g_1,\ldots,g_{d-1}) \mapsto \big( g_1^{-T},\ldots,g_{d-1}^{-T} \big)
	\end{align*}
the representations $( \varrho^* \circ \varphi ) \otimes \prid_{n-k}$ and $\varrho \otimes \prid_{n-k}$ of $G \times \prGL_{n-k}(\prC)$ are isomorphic. Therefore,  $(V^* \otimes \prC^{n-k}, \varrho^* \otimes \prid_{n-k})$ is prehomogeneous if and only if $(V \otimes \prC^{n-k}, \varrho \otimes \prid_{n-k})$ is prehomogeneous, which ends the proof.
\end{proof}

The previous theorem leads us to the definition of a castling transform. Castling transform finds its best motivation as the reflection functor for quiver representation theory introduced by Bernstein, Gelfand and Ponomarev in \cite{prBernstein1973coxeter}. It was
studied also in \cite{prSatoKimura} for classifying prehomogeneous spaces. Our presentation mainly follows \cite{prVenturelli} and \cite{prManivelLectures} and avoids the language of quivers.

\begin{defn}
Let $(a_1,\ldots, a_d), (b_1, \ldots, b_d) \in \prN^d$ and denote the symmetric group of $\lbrace 1,\ldots, d\rbrace$ by $\mathfrak{S}_d$.
	\begin{numlist}
	\item We say $(b_1, \ldots, b_d)$ is a \textbf{castling transform} of $(a_1,\ldots,a_d)$, if there exists $\sigma \in \mathfrak{S}_d$ such that
	\begin{align*}
	b_d = \prod_{i=1}^{d-1} a_{\sigma(i)} - a_{\sigma(d)} \quad \text{and} \quad b_j = a_{\sigma(j)}
	\end{align*}
	for $1\leq j \leq d-1$.
	\item The tuples $(a_1,\ldots, a_d)$ and $(b_1, \ldots, b_d)$ are said to be \textbf{castling equivalent}, in symbols $(a_1,\ldots, a_d) \sim (b_1, \ldots, b_d)$, if $(b_1, \ldots, b_d)$ results from permuting the $a_i$ or $(b_1, \ldots, b_d)$ results from $(a_1,\ldots, a_d)$ by applying a finite number of castling transformations to $(a_1,\ldots, a_d)$.
	\end{numlist}
\end{defn}

\begin{remark}\label{pr:remCastlingEquiv}
We note the following about castling transformations.
	\begin{numlist}
		\item The castling equivalence is indeed an equivalence relation on $\prN^d$.  By definition,  $(a_1,\ldots,a_d) \sim (a_{\sigma(1)}, \ldots, a_{\sigma(d)})$ for all $\sigma \in \mathfrak{S}_d$. Thus, we may always permute the entries, e.g., to obtain non-decreasing order,  without leaving the castling equivalence class.
		
		\item As a consequence of Theorem~\ref{pr:thmCastling}, castling equivalence preserves prehomogeneity. That is, if $(a_1,\ldots,a_d) \sim (b_1,\ldots,b_d)$ then $\prC^{a_1} \otimes \cdots \otimes \prC^{a_d}$ is prehomogeneous if and only if $\prC^{b_1} \otimes \cdots \otimes \prC^{b_d}$ is prehomogeneous.
		
		\item In \cite[Theorem~1.2]{prBryan2018existence}  the dimension of the generic orbit for any tensor format $(a_1,\ldots, a_d)$ is computed with a tricky arithmetic function.
		
		\item Castling transforms are also of interest in invariant theory and, recently, in algebraic statistics. For the natural action of $\prSL_{a_1}(\prC) \times \cdots \times \prSL_{a_d}(\prC)$ on $\prC^{a_1} \otimes \cdots \otimes \prC^{a_d}$ castling transforms preserve stability notions from invariant theory. This fact is -- for different stability notions -- (implicitly) contained in \cite{prElashvili72, prPopov70, prSatoKimura} and it is made explicit for all notions in \cite{prDMW20}.  Moreover, the authors of \cite{prDMW20} crucially use that castling transforms preserve stability notions, together with relations between invariant theory and algebraic statistics (\cite{prAKRS20}), to determine the maximum likelihood thresholds of tensor normal models; see \cite[Theorem~1.1]{prDMW20}. 
	\end{numlist}
\end{remark} 

In addition to preserving prehomogeneity, castling equivalence is suited for the study of prehomogeneous spaces thanks to the following two lemmas. 

\begin{lem}
$N(a_1,\ldots,a_d)$ is invariant under castling equivalence.
\end{lem}

\begin{proof}
Let $(a_1,\ldots, a_d) \in \prN^d$.  Clearly, $N(a_1,\ldots,a_d) = N(a_{\sigma(1)},\ldots,a_{\sigma(d)})$ holds for all $\sigma \in \mathfrak{S}_d$ by commutativity of addition and multiplication.  Thus, it is enough to show $N(a_1,\ldots,a_d) = N(b_1,\ldots,b_d)$ in the case that $b_i = a_i$ for $i < d$ and $b_d = a_1 a_2 \cdots a_{d-1} - a_d > 0$. We compute
	\begin{align*}
	N(b_1,\ldots,b_d) &= 1-d + \left[ \sum_{i=1}^d b_i^2 \right] - \left[ \prod_{i=1}^d b_i \right] \\
	&= 1-d + \left[ \sum_{i=1}^{d-1} a_i^2 + \prod_{i=1}^{d-1} a_i^2 - 2 \prod_{i=1}^d a_i + a_d^2 \right] - \left[ \prod_{i=1}^{d-1} a_i^2 - \prod_{i=1}^d a_i \right] \\
	&= 1-d + \sum_{i=1}^d a_i^2 - \prod_{i=1}^d a_i = N(a_1,\ldots,a_d)
	\end{align*}
as required.
\end{proof}

\begin{lem}\label{pr:lemMinCastling}
Given $(a_1,\ldots,a_d) \in \prN^d$ there is at most one castling transform $(b_1,\ldots,b_d)$, up to permuting the $b_i$, such that $a_1 \cdots a_d > b_1 \cdots b_d$.
\end{lem}

\begin{proof}
After eventually permuting the $a_i$, we consider for
	\begin{align*}
	u := \prod_{i=3}^{d} a_i, \quad a_1':= ua_2 - a_1, \quad a_2':= u a_1 -a_2 
	\end{align*}
the castling transforms $(a_1',a_2,a_3,\ldots,a_{d}), (a_1,a_2',a_3,\ldots,a_d) \in \prN^d$ of $(a_1,\ldots,a_d)$ and assume
	\begin{equation}\label{pr:eqMinCastling}
	a'_1 = ua_2-a_1<a_1 \quad \text{ and } \quad a'_2 = ua_1-a_2< a_2 \, ,
	\end{equation}
which means $u a_2 < 2a_1$ and $u a_1 < 2 a_2$ respectively.
Therefore, we get $u^2 a_1 a_2 < 4 a_1 a_2$ and as $u$ is a positive integer we necessarily have $u=1$. The latter implies
$a_1' = -a_2'$ and with $a_1', a_2' \geq 0$ we deduce $a_1'=a_2'=0$, which is a contradiction.
\end{proof}

We say that $(a_1,\ldots,a_d) \in \prN^d$ is a \textbf{\textit{minimal element}}, if $a_1 \cdots a_d \leq b_1 \cdots b_d$ for all $(b_1,\ldots,b_d) \in \prN^d$ with $(a_1,\ldots,a_d) \sim (b_1,\ldots,b_d)$. With this definition we obtain the following corollary of Lemma~\ref{pr:lemMinCastling}.

\begin{cor}
Every castling equivalence class in $\prN^d$ contains, up to permutation of the entries, a unique minimal element.
\end{cor}

Hence, for deciding if $\prC^{a_1} \otimes \cdots \otimes \prC^{a_d}$ is prehomogeneous, it suffices to know whether the unique minimal element in the castling class of $(a_1,\ldots,a_d)$ gives rise to a prehomogeneous space. Note, that by Lemma~\ref{pr:lemMinCastling} there is (up to permutation) a \emph{unique} way for computing the minimal element by castling transformations.

The described approach was first used by Sato and Kimura in \cite{prSatoKimura} to classify the prehomogeneous tensor spaces. Venturelli refined these results in his work \cite{prVenturelli} and below we state the main result of his work. It comes with the advantage that one only has to apply castling transforms if $N(a_1,\ldots,a_d)=2$.

\begin{thm}[Sato-Kimura and Venturelli]\label{pr:thmSatoKimuraVenturelli} \ \\
Let $d \geq 3$ and $(a_1,\ldots,a_d) \in \prN^d$.
	\begin{numlist}
	\item If $N(a_1,\ldots,a_d) \leq -1$, then $\prC^{a_1} \otimes \cdots \otimes \prC^{a_d}$ is \emph{not} prehomogeneous.
	\item If $N(a_1,\ldots,a_d) \in \lbrace 0,1 \rbrace$, then $\prC^{a_1} \otimes \cdots \otimes \prC^{a_d}$ is prehomogeneous.
	\item If $N(a_1,\ldots,a_d) = 2$, then $(a_1,\ldots,a_d)$ is castling equivalent to \emph{either}
		\begin{numlist}
		\item a minimal tuple $(b_1,\ldots,b_d)$ with $2 \leq b_{d-3} \leq \ldots \leq b_d$ and $d \geq 4$. In this case $\prC^{a_1} \otimes \cdots \otimes \prC^{a_d}$ is prehomogeneous.
		
		\item \emph{or} to a minimal tuple of the form $(1,\ldots,1,2,k,k)$ for a unique $k \in \prN$. Here, $\prC^{a_1} \otimes \cdots \otimes \prC^{a_d}$ is prehomogeneous if and only if $k \leq 3$.
		\end{numlist}
		
	\item If $N(a_1,\ldots,a_d) \geq 3$, then $\prC^{a_1} \otimes \cdots \otimes \prC^{a_d}$ is prehomogeneous.
	\end{numlist}
\end{thm}

\begin{proof}
This is \cite[Theorem~1]{prVenturelli} and for a proof we refer the reader to the corresponding article.
\end{proof}

%%%%%%% End: Group Action on Tensors %%%%%%%

%%%%%%% Begin: Tensor Rank and Border Rank %%%%%%%

\section{Tensor Rank and Border Rank}\label{pr:secRank}

We continue the study of the natural $G = \prGL(V_1) \times \cdots \times \prGL(V_d)$ action on $V = V_1 \otimes \cdots \otimes V_d$. For this, we introduce several notions of rank for tensors and illustrate in the case $\prC^2 \otimes \prC^2 \otimes \prC^2$ how one may study the $G$-orbits with these notions. Moreover, tensor (border) rank measures the (approximate) complexity. This will be beneficial in the following sections, especially when talking about the complexity of matrix multiplication. Besides complexity, the notions of tensor rank have many important applications, see \cite{prLandsbergBook}.

Considering Lemma~\ref{pr:lemSec1} b) there is a natural generalization of the concept of \emph{matrix rank}, the so-called \emph{tensor rank}.

\begin{defn}\label{pr:defnTensorRank}
Let $t \in V= V_1 \otimes \cdots \otimes V_d$ and $s \in W = W_1 \otimes \cdots \otimes W_d$.
We call $t$ \textbf{decomposable} if there are $v_i \in V_i$ such that $t = v_1 \otimes \cdots \otimes v_d$. The \textbf{rank} of $t$ is defined as
	\begin{align*}
	\prrk(t) := \min \left\lbrace r \in \prZ_{\geq 0} \;\Big\vert \; t = \sum_{i=1}^r t_i \, , \; t_i \in V_1 \otimes \cdots \otimes V_d \text{ decomposable} \right\rbrace \, .
	\end{align*}
Moreover, we say $s$ is a \textbf{restriction} of $t$, in symbols $s \leq t$, if there are linear maps $\alpha_i \colon V_i \to W_i$ such that $(\alpha_1 \otimes \cdots \otimes \alpha_d)(t) = s$.
\end{defn}

\begin{remark}\label{pr:remTensorRk} Let $t \in V= V_1 \otimes \cdots \otimes V_d$ and $s \in W = W_1 \otimes \cdots \otimes W_d$ be tensors.
	\begin{abclist}
	\item The tensor $t$ is decomposable if and only if $\prrk(t)\leq 1$. Hence, non-zero decomposable tensors are exactly the tensors of rank one.
	\item The natural action of $ G= \prGL(V_1) \times \cdots \times \prGL(V_d)$ preserves the rank of tensors. In contrast to the matrix case, there may be several distinct $G$-orbits for the same tensor rank. For example, even in the simple case $\prC^2 \otimes \prC^2 \otimes \prC^2$ with finitely many orbits there are four orbits for tensor rank two, compare Table~\ref{pr:tabRepresentativesC2} below.
	\item Since $V$ has a basis that only consists of decomposable tensors, we have the inequality $\prrk(t) \leq \dim V = a_1 \cdots a_d$, where $a_i = \dim V_i$. This inequality can be easily improved to $\prrk(t) \leq \prod_{j\neq i}a_j$ for any
$i=1,\ldots, d$, by expanding the expression of $t$ in the basis of $V$ obtained by the tensor product of the basis of each $V_i$,
compare with Exercise~1d).
 %allows to group together all elements sharing
%the same $d-1$ basis elements corresponding to $d-1$ factors.

	\item Let $\langle r \rangle := \sum_{i=1}^r e_i \otimes e_i \otimes e_i \in (\prC^r)^{\otimes 3}$ be the tensor of componentwise multiplication in $\prC^r$, where $e_1,\ldots,e_r$ denotes the standard basis of $\prC^r$. Then the rank of $t \in V_1\otimes V_2\otimes V_3$ is characterized by
	\begin{align*}
	\prrk(t) \leq r \quad \Longleftrightarrow \quad t \leq \langle r \rangle,
	\end{align*}
where $\dim V_i$ may be smaller or larger than $r$.
	Consequently, $s \leq t$ implies $\prrk(s) \leq \prrk(t)$, because the notion of restriction is transitive. Moreover, this discussion can be generalized to any number of factors $d$.
	\item Although the rank of a tensor is computable (e.g., by brute force), this is in general very inefficient.
	Indeed, for $d \geq 3$ the computation of tensor rank over $\prQ$, $\prR$ and $\prC$ is NP-hard, see \cite{prHillarLim}.	
	\end{abclist}
\end{remark}

Part~b) of the preceding remark shows that tensor rank is not enough for classifying the orbits, even if there are only finitely many. Even worse, also the naive attempt
	\begin{align*}
	X_r := \big\lbrace t \in V_1 \otimes \cdots \otimes V_d \mid \prrk(t) \leq r \big\rbrace
	\end{align*}
does not yield Zariski-closed sets in general, compare Example~\ref{pr:exWstate}. At least $X_1$ is always Zariski-closed, because it is the affine cone of the Segre variety $\prP(V_1) \times \cdots \times \prP(V_d) \subseteq \prP(V)$, compare Remark~\ref{pr:remTensorRk}~a). One way out is to enforce Zariski-closedness by definition. This leads to the \emph{border rank} of a tensor. This notion was used implicitly already in \cite{prBini2-7799} and the term first appeared in the article \cite{prBiniLottiRomani80}.

\begin{defn}\label{pr:defnBorderRk}
Let $t \in V_1 \otimes \cdots \otimes V_d$ and denote the Zariski-closure of $X_r$ by $\overline{X_r}$. The \textbf{border rank} of $t$ is defined as
 \begin{align*}
 \prbrk(t) := \min \left\lbrace r \in \prZ_{\geq 0} \mid t \in \overline{X_r} \right\rbrace = \min \big\lbrace r \mid t \text{ is a limit of rank } r \text{ tensors} \big\rbrace \, .
 \end{align*}
\end{defn}

\begin{remark}\label{pr:remZariskiVsEuclidean}
Since $X_r$ is Zariski-open in its Zariski-closure $\overline{X_r}$, the Zariski-closure of $X_r$ equals the Euclidean closure of $X_r$, compare, e.g., \cite[Section~3.1.6]{prLandsbergComplexity17}. Hence, in the above definition we may consider \emph{``t is a limit of rank $r$ tensors''} with respect to the Euclidean topology.
\end{remark}

By construction, $\underline{X}_r := \lbrace t \in V \mid \prbrk(t) \leq r \rbrace$ is an algebraic variety, and similarly to equation~\eqref{pr:eqXrMatrixCase} from the case of matrices, we get an ascending chain
	\begin{align*}
	\lbrace 0 \rbrace = \underline{X}_0 \subseteq \underline{X}_1 \subseteq \underline{X}_2 \subseteq \ldots \subseteq \underline{X}_{\dim V} = V \, .
	\end{align*}
In fact, if $S:= \prP(V_1) \times \cdots \times \prP(V_d) \subseteq \prP(V)$ denotes the Segre variety, then it can be proven that $\underline{X}_r$ is the affine cone of the $r$-th \emph{secant variety} of $S$. Thus, from the algebraic geometry point of view border rank is more convenient than tensor rank. For further details on the geometry of rank and border rank we refer to \cite[chapter~5]{prLandsbergBook}.

Of course, we always have $\prbrk(t) \leq \prrk(t)$ and this inequality may be strict.

\begin{ex}\label{pr:exWstate}
Let $\lbrace e_0, e_1 \rbrace$ be a basis of $\prC^2$. The tensor
	\begin{align*}
	w := e_0 \otimes e_0 \otimes e_1 + e_0 \otimes e_1 \otimes e_0 + e_1 \otimes e_0 \otimes e_0 \in \prC^2 \otimes \prC^2 \otimes \prC^2
	\end{align*}
has at most rank three. Actually, one has $\prrk(w)=3$, compare Exercise~2'(6). On the other hand, for $\prEps >0$
	\begin{equation}\label{pr:eqWstate}
	(e_0 + \prEps e_1) \otimes (e_0 + \prEps e_1) \otimes (e_0 + \prEps e_1) - e_0 \otimes e_0 \otimes e_0
	= \prEps w + \prEps^2 t_1 + \prEps^3 t_2
	\end{equation}
for some $t_1,t_2 \in \prC^2 \otimes \prC^2 \otimes \prC^2$. Equation~\eqref{pr:eqWstate} shows that
	\begin{align*}
	w = \lim_{\prEps \to 0} \, \frac{1}{\prEps} \big[ (e_0 + \prEps e_1) \otimes (e_0 + \prEps e_1) \otimes (e_0 + \prEps e_1) - e_0 \otimes e_0 \otimes e_0 \big] \, ,
	\end{align*}
i.e., $\omega$ is a limit of rank two tensors. Thus, $\prbrk(w) \leq 2 < 3 = \prrk(w)$. In fact $\prbrk(w)=2$, because the set of tensors of rank at most one is Zariski-closed.\hfill\prExample
\end{ex}

Considering equation~\eqref{pr:eqWstate} may motivate another way of defining the border rank, that will be needed in Section~\ref{pr:secComplexity}. Actually, this leads to a characterization of border rank that is the counterpart of describing the tensor rank via restriction as in Remark~\ref{pr:remTensorRk}~d). For simplicity we work with tensors of order three.

\begin{defn}
Let $d = 3$, $t \in V_1 \otimes V_2 \otimes V_3$ and let $\pre$ be an indeterminate over $\prC$. We say $t$ can be \textbf{approximated with degree $\mathbf{q}$ by tensors of rank $\mathbf{r}$}, in symbols $t \unlhd_q \langle r \rangle$, if there exist vectors $v_{i,\rho}(\pre) \in \prC[\pre] \otimes_{\prC} V_i$ for $i=1,2,3$ and $\rho = 1,\ldots,r$, such that
	\begin{equation}\label{pr:eqDegeneration}
	\pre^q t + \pre^{q+1}s(\pre) = \sum_{\rho=1}^r v_{1,\rho}(\pre) \otimes v_{2,\rho}(\pre) \otimes v_{3,\rho}(\pre) \, ,
	\end{equation}
where $s \in \prC[\pre] \otimes_{\prC} (V_1 \otimes V_2 \otimes V_3)$. (The notation $t \unlhd_q \langle r \rangle$ comes from the concept of \emph{degeneration}, and $\langle r \rangle$ denotes again the tensor $\sum_i e_i \otimes e_i \otimes e_i \in (\prC^r)^{\otimes 3}$.)
\end{defn}

Note that compared to \cite{prAlgComplTheoryBook} and \cite{prStrRelBilinearComplexity87} our definition shifts the role of $q$ by one. In \cite[Section~15.4]{prAlgComplTheoryBook} the border rank of a tensor $t$ is defined as
	\begin{equation}\label{pr:eqSecondDefBorderRk}
	\min \big\lbrace r \mid \exists \, q \colon \: t \unlhd_q \langle r \rangle \big\rbrace \, .
	\end{equation}
Replacing in equation~\eqref{pr:eqDegeneration} the variable $\pre$ by some $\prEps > 0$ and dividing by $\prEps^q$ we see that $\prbrk(t) \leq r$ whenever $t \unlhd_q \langle r \rangle$. Hence, we obtain $\prbrk(t) \leq \min \big\lbrace r \mid \exists \, q \colon \: t \unlhd_q \langle r \rangle \big\rbrace$ for our definition of border rank. Actually, also the converse inequality holds, i.e., both definitions coincide. An explanation can be found in \cite[Section~20.6]{prAlgComplTheoryBook} or in \cite[Sections~4 and 5]{prStrRelBilinearComplexity87}.\footnote{More precisely, at the beginning of \cite[Section~4]{prStrRelBilinearComplexity87} Strassen defines border rank as in Definition~\ref{pr:defnBorderRk} and states that $t$ is a degeneration of the tensor $\langle \prbrk(t) \rangle$, in symbols $t \unlhd \langle \prbrk(t) \rangle$. Combining Lemma~5.5 and Theorem~5.8 from \cite{prStrRelBilinearComplexity87} shows the equivalence of the definitions.}

The advantage of defining border rank as in \eqref{pr:eqSecondDefBorderRk} is, that it allows to \emph{switch from approximate to exact algorithms}. This is made precise in the following lemma, which will be helpful in Section~\ref{pr:secComplexity}.

\begin{lem}\label{pr:lemApproxToExact}
Let $t \in V_1 \otimes V_2 \otimes V_3$. If $t \unlhd_q \langle r \rangle$, then $\prrk(t) \leq (q+1)^2 r$.
\end{lem}

\begin{proof}
Let $t \unlhd_q \langle r \rangle$ and assume a representation as in equation~\eqref{pr:eqDegeneration}. Then we can write
	\begin{align*}
	v_{1,\rho}(\pre) = \sum_{\lambda \geq 0} v_{1,\rho}^{(\lambda)} \pre^{\lambda} \, , \quad
	v_{2,\rho}(\pre) = \sum_{\mu \geq 0} v_{2,\rho}^{(\mu)} \pre^{\mu} \, , \quad
	v_{3,\rho}(\pre) = \sum_{\nu \geq 0} v_{3,\rho}^{(\nu)} \pre^{\nu} \, ,
	\end{align*}
where $v_{1,\rho}^{(\lambda)} \in V_1$, $v_{2,\rho}^{(\lambda)} \in V_2$ and $v_{3,\rho}^{(\lambda)} \in V_3$, only finitely many being non-zero. Multiplying the right hand side of equation~\eqref{pr:eqDegeneration} out and comparing the coefficient for $\pre^q$ gives
	\begin{align*}
	t = \sum_{\rho=1}^r \sum_{\substack{\lambda,\mu,\nu \geq 0 \\ \lambda + \mu + \nu = q}} v_{1,\rho}^{(\lambda)} \otimes v_{2,\rho}^{(\mu)} \otimes v_{3,\rho}^{(\nu)} \, .
	\end{align*}
Combining this with the fact that
	\begin{align*}
	\big\vert \left\lbrace (\lambda,\mu,\nu) \in \prZ_{\geq 0}^3 \mid \lambda + \mu + \nu = q \right\rbrace \big\vert = {{q+2}\choose{2}} = \frac{(q+2)(q+1)}{2} \leq (q+1)^2
	\end{align*}
we conclude that $\prrk(t) \leq r(q+1)^2$.
\end{proof}

In the following we present parts of the material from the exercise sessions of the fall school. These fit nicely here and form a natural extension of this section. First we introduce the multilinear rank and solve part~a) and b) of Exercise~1. As an application, we see afterwards how multilinear rank and tensor rank together classify the orbits of the natural $\prGL_2(\prC) \times \prGL_2(\prC) \times \prGL_2(\prC)$ action on $\prC^2 \otimes \prC^2 \otimes \prC^2$. Thereby, we solve Exercises~2 and 2' partly and emphasize the geometric ideas. \textit{(Those who wish to solve the exercises on their own, are encouraged to jump to Section~\ref{pr:secComplexity} \emph{after} reading Definition~\ref{pr:defnMultilinearRk}.)}

For simplicity we work with $3$-tensors and assume $V = \prC^a \otimes \prC^b \otimes \prC^c$. Let $U$ and $W$ be $\prC$-vector spaces, denote the dual of $U$ by $U^\vee$ and the vector space of linear maps from $U^\vee$ to $W$ by $\prHom(U^\vee, W)$. Recall the canonical isomorphism
	\begin{align*}
	U \otimes W \to \prHom \left(U^\vee, W \right), \quad u \otimes w \; \mapsto \; \big( f \mapsto f(u)w \big) \, ,
	\end{align*}
which is given on decomposable tensors.

%For a $\prC$-vector space $W$ we denote its dual by $W^\vee$. Moreover, for $i \in \lbrace 1,\ldots,d \rbrace$ we set
%	\begin{align*}
%	V_{\hat{i}} := V_1 \otimes \cdots \otimes V_{i-1} \otimes V_{i+1} \otimes \cdots \otimes V_d
%	\end{align*}
%for omitting $V_i$ in the tensor product. Given bases $\lbrace v_j^{(i)} \mid j=1,\ldots,a_i \rbrace$ of $V_i$ and a tensor $t \in V$ we can write
%	\begin{align*}
%	t = \sum_{j_1,\ldots,j_d} t_{j_1,\ldots, j_d} \, v_{j_1}^{(1)} \otimes \cdots \otimes v_{j_d}^{(d)}
%	\end{align*}
%with $t_{j_1,\ldots, j_d} \in \prC$.
%
%The \textbf{$\mathbf{i}$-th contraction} of $t \in V$ is the linear function
%	\begin{align*}
%	\Gamma_i(t) \colon (V_i)^\vee &\to V_{\hat{i}} \\ f &\mapsto \sum_{j_1,\ldots,j_d} f \left(v_{j_i}^{(i)} \right) t_{j_1,\ldots, j_d} \, v_{j_1}^{(1)} \otimes \cdots \otimes v_{j_{i-1}}^{(i-1)} \otimes v_{j_{i+1}}^{(i+1)} \otimes \cdots \otimes v_{j_d}^{(d)}
%	\end{align*}

\begin{defn}\label{pr:defnMultilinearRk}
A tensor $t \in V$ corresponds under the canonical isomorphism
	\begin{align*}
	V \cong \prHom \left( \left( \prC^a \right)^\vee, \prC^b \otimes \prC^c \right) \quad \text{ to } \quad \Gamma_1(t) \colon \left( \prC^a \right)^\vee \to \prC^b \otimes \prC^c \, ,
	\end{align*}
which is a linear map and called the first \textbf{contraction} (or \textbf{flattening}) of $t$. Similarly, one defines $\Gamma_2(t) \colon \left( \prC^b \right)^\vee \to \prC^a \otimes \prC^c$ and $\Gamma_3(t) \colon \left( \prC^c \right)^\vee \to \prC^a \otimes \prC^b$, the second and third contraction respectively.

The rank of $\Gamma_{l}(t)$ is denoted $r_l(t)$ and the tuple $\big( r_1(t),r_2(t),r_3(t) \big)$ is called the \textbf{multilinear rank} of $t$.
\end{defn}

\begin{remark}\label{pr:remMultilinearRkInvariant}
The natural $\prGL_a(\prC) \times \prGL_b(\prC) \times \prGL_c(\prC)$ action on $V$ preserves multilinear rank, because $\prGL_n(\prC) \times \prGL_m(\prC)$ is a subgroup of $\prGL(\prC^n \otimes \prC^m)$ via $(g_1,g_2) \mapsto g_1 \otimes g_2$.
\end{remark}

\begin{ex}\label{pr:exContractionMatrix}
Let us write for $t \in V$ the first contraction in coordinates. To do so, choose bases $\lbrace a_i \rbrace_i$, $\lbrace b_j \rbrace_j$ and $\lbrace c_k \rbrace_k$ of $\prC^a$, $\prC^b$ and $\prC^c$ respectively. Then 
\begin{align*}
	t = \sum_{i,j,k} t_{ijk} \, a_i \otimes b_j \otimes c_k
	\end{align*}
with $t_{ijk} \in \prC$. Now, the first contraction of $t$ is 
	\begin{align*}
	\Gamma_1(t) \colon \left( \prC^a \right)^\vee \to \prC^b \otimes \prC^c, \quad f \mapsto \sum_{i,j,k} f(a_i) t_{ijk} \: b_j \otimes c_k \, .
	\end{align*}
If $\lbrace \alpha_i \rbrace_i$ is the dual basis of $\lbrace a_i \rbrace_i$, then $\Gamma_1(t)$ is represented by the matrix
	\begin{equation}\label{pr:eqContractionMatrix}
	[t_{ijk}]_{i,(j,k)} = 
	\begin{pmatrix}
	t_{111} & t_{112} & \ldots & t_{11c} & t_{121} & \ldots & \ldots & t_{1bc} \\ 
	t_{211} & t_{212} & \ldots & t_{21c} & t_{221} & \ldots & \ldots & t_{2bc} \\ 
	\vdots & \vdots &  & \vdots & \vdots & & & \vdots \\ 
	t_{a11} & t_{a12} & \ldots & t_{a1c} & t_{a21} & \ldots & \ldots & t_{abc}
	\end{pmatrix} \in \prC^{a \times bc}
	\end{equation}
where $(\alpha_1,\ldots,\alpha_a)$ and $(b_1 \otimes c_1,\ldots,b_1 \otimes c_c,b_2 \otimes c_1, \ldots \ldots, b_b \otimes c_c)$ are the ordered bases, which are considered for this matrix presentation.\hfill\prExample
\end{ex}

The proof of the next proposition solves part a) and b) of Exercise 1.

\begin{prop}\label{pr:propMultilinearRk}
Let $t \in \prC^a \otimes \prC^b \otimes \prC^c$. Then $\prrk(t) = 1$ if and only if $t$ has multilinear rank $(1,1,1)$. In fact, already $r_i(t) = r_j(t) = 1$ for some $i,j \in \lbrace 1,2,3 \rbrace$ with $i \neq j$ implies $\prrk(t)=1$.
\end{prop}

\begin{proof}
If $\prrk(t)=1$, then there are non-zero vectors $a_1 \in \prC^a$, $b_1 \in \prC^b$, $c_1 \in \prC^c$ such that $t = a_1 \otimes b_1 \otimes c_1$. Extending these vectors to bases $\lbrace a_i \rbrace_i$, $\lbrace b_j \rbrace_j$ and $\lbrace c_k \rbrace_k$ we see that, similar to Example~\ref{pr:exContractionMatrix}, all three contractions have
	\begin{align*}
	\begin{pmatrix} 1 & 0 & \ldots & 0 \\ 0 & 0 & \ldots & 0 \\ \vdots & \vdots & & \vdots \\ 0 & 0 & \ldots & 0 \\
	\end{pmatrix}
	\end{align*}
as a matrix representation. Therefore, $t$ has multilinear rank $(1,1,1)$.

Conversely, assume without loss of generality that $1 = r_1(t) = r_2(t)$. Choose bases $\lbrace a_i \rbrace_i$, $\lbrace b_j \rbrace_j$ and $\lbrace c_k \rbrace_k$ of $\prC^a$, $\prC^b$ and $\prC^c$ respectively. Again, we denote the coordinates of $t$ by $t_{ijk}$ and, after reordering the bases, we may assume $t_{111} \neq 0$. Since $r_1(t) = 1$, all rows of the matrix in \eqref{pr:eqContractionMatrix} are linearly dependent. Using that the first row is non-zero ($t_{111} \neq 0$), we obtain $\lambda_i \in \prC$ (possibly zero) for $i \geq 2$ such that
	\begin{align*}
	\forall j \in \lbrace 1,\ldots,b \rbrace : \forall k \in \lbrace 1,\ldots,c \rbrace : \quad \lambda_i t_{1jk} = t_{ijk} \, .
	\end{align*}
Setting $\lambda_1 := 1$ we get the latter property also for $i=1$.
Similarly, $r_2(t)=1$ gives $\mu_j \in \prC$ for $j \geq 1$ ($\mu_1=1$) such that $\mu_{j} t_{i1k} = t_{ijk}$ for all $i$ and all $k$. Now, we compute
	\begin{align*}
	t &= \sum_{i,j,k} t_{ijk} \, a_i \otimes b_j \otimes c_k =\sum_{i,j,k} \lambda_i t_{1jk}\, a_i \otimes b_j \otimes c_k\\ 
	&= \sum_{j,k} t_{1jk} \left( \sum_i \lambda_i a_i \right) \otimes b_j \otimes c_k
	= \sum_{j,k} \mu_j t_{11k} \left( \sum_i \lambda_i a_i \right) \otimes b_j \otimes c_k \\
	&= \left( \sum_i \lambda_i a_i \right) \otimes \left( \sum_j \mu_j b_j \right) \otimes \left( \sum_k t_{11k} c_k \right) \, ,
	\end{align*}
which shows that $t$ is decomposable. Finally, $t \neq 0$ yields $\prrk(t)=1$.
\end{proof}

%\begin{prop}
%Let $t \in V$. Then
%	\begin{align*}
%	r_1(t) \leq \prbrk(t) \leq \prrk(t) \leq r_j(t) r_k(t)
%	\end{align*}
%for all $i,j,k$ such that $\lbrace i,j,k \rbrace = \lbrace 1,2,3 \rbrace$.
%\end{prop}
%
%\begin{proof}
%
%\end{proof}

\begin{remark}
The notion of multilinear rank generalizes to $d$-tensors and is a tuple of the form $\big( r_1(t),\ldots,r_d(t) \big)$, where the $r_i(t)$ are the ranks of the $d$ contractions defined by $t \in V_1 \otimes \cdots \otimes V_d$. Furthermore, also the statement of Proposition~\ref{pr:propMultilinearRk} generalizes to $d$-tensors as may already be clear from the proof. 
\end{remark}

%%%%%%% Extended Example %%%%%%%%%%%

\subsection{Extended Example on $\prC^2 \otimes \prC^2 \otimes \prC^2$}
\label{pr:secDim2}

Now, we study the natural action of $G := \prGL_2(\prC) \times \prGL_2(\prC) \times \prGL_2(\prC)$ on $\prC^2 \otimes \prC^2 \otimes \prC^2$. Thereby, we offer an outline for solving Exercises~2 and 2'. Not everything will be proven, but a complete overview to all results of these two exercises is provided by Table~\ref{pr:tabOrbitsProj}, Table~\ref{pr:tabRepresentativesC2} and Figure~\ref{pr:figOrbitsProjC2}.

To stress which tensor factor is meant, we set $V := \prC^a \otimes \prC^b \otimes \prC^c$ for $a,b,c=2$. Moreover, we fix bases $\lbrace a_0, a_1 \rbrace$, $\lbrace b_0, b_1 \rbrace$ and $\lbrace c_0, c_1 \rbrace$ of $\prC^a$, $\prC^b$ and $\prC^c$ respectively. Often we will suppress the tensor product to increase readability, i.e., instead of $a_0 \otimes b_0 \otimes c_0$ we write $a_0 b_0 c_0$. Whenever we need to consider an ordered basis of $\prC^a \otimes \prC^b$ (e.g., for a matrix representing a contraction of a tensor), we use $(a_0b_0,a_0b_1,a_1b_0,a_1b_1)$ and similarly for $\prC^a \otimes \prC^c$, $\prC^b \otimes \prC^c$.

First, we discuss how many $G$-orbits there are. Afterwards we present a classification of all non-zero orbits and give some geometric ideas and intuition. For this, it will be more convenient to work with the induced $G$-action on $\prP(V)$.

\medskip

Remember that tensor rank and multilinear rank are $G$-invariant. Hence, we may ask how many orbits there are for the possible ranks, respectively multilinear ranks. 
%Since the latter is a finer measure than tensor rank, we stick to the multilinear rank. 
Of course, to rank $0$ (and multilinear rank $(0,0,0)$) corresponds just one orbit, the zero orbit. Besides this, there are a priori eight possible multilinear ranks, namely $(i,j,k)$ for $i,j,k \in \lbrace 1,2 \rbrace$. However, Proposition~\ref{pr:propMultilinearRk} shows that $(1,1,2)$, $(1,2,1)$ and $(2,1,1)$ are not admissible, because $r_i(t) = r_j(t)=1$ for $i\neq j$ enforces multilinear rank $(1,1,1)$. Thus, we are left with five tuples, which indeed happen to be multilinear ranks of some tensors in $V$.

The rank one tensor $a_0 b_0 c_0$ has multilinear rank $(1,1,1)$ and the rank two tensor $a_0 b_0 c_0 + a_1 b_1 c_1$ has multilinear rank $(2,2,2)$. Moreover, the first, second and third contraction of $a_0b_0c_0 + a_0 b_1 c_1$ have representing matrix
	\begin{align*}
	\begin{pmatrix}
	1 & 0 & 0 & 1 \\ 0 & 0 & 0 & 0
	\end{pmatrix} \, , \quad
	\begin{pmatrix}
	1 & 0 & 0 & 0 \\ 0 & 1 & 0 & 0
	\end{pmatrix}  \quad \text{ and } \quad
	\begin{pmatrix}
	1 & 0 & 0 & 0 \\ 0 & 1 & 0 & 0
	\end{pmatrix}
	\end{align*}
respectively. (Recall: The ordered basis for $\prC^b \otimes \prC^c$ is $(b_0c_0, b_0c_1, b_1c_0, b_1c_1)$ giving the matrix on the left.) Thus, the tensor $a_0b_0c_0 + a_0 b_1 c_1$ has multilinear rank $(1,2,2)$. Analogously, $a_0b_0c_0 + a_1 b_0 c_1$ and $a_0b_0c_0 + a_1 b_1 c_0$ have multilinear ranks $(2,1,2)$ and $(2,2,1)$ respectively.

We can conclude that there are at least five non-zero orbits. Actually, only the multilinear rank $(2,2,2)$ gives rise to more than one orbit. Let us first argue, why the other multilinear ranks correspond to exactly one orbit. Clearly, a rank one tensor $x  y  z \in V$ is contained in the orbit of $a_0b_0c_0$, since there are invertible linear maps sending $a_0$ to $x$, $b_0$ to $y$ and $c_0$ to $z$. Thus, the rank one tensors form one orbit.

Given $t \in V$ of multilinear rank $(1,2,2)$ we deduce with $\prrk(t) \leq r_1(t)r_2(t) = 2$ (Exercise~1d)) that $t$ has rank two. Therefore, we may write $t = x_0y_0z_0 + x_1y'_1z_1$ for some $x_i \in \prC^a$, $y_0, y'_1 \in \prC^b$ and $z_i \in \prC^c$. Now, note that $r_1(t) = 1$ implies that $x_0$ and $x_1$ are linearly dependent.  Hence, we may write $t = xy_0z_0 + x y_1z_1$ with $x := x_0 \neq 0$ and $y_1 := \lambda y'_1$, where $\lambda \in \prC^\times$ is such that $x_1 = \lambda x_0$. On the other hand, the $y_i$ and the $z_i$ form a basis of $\prC^2$ by $r_2(t)=r_3(t)=2$. Thus, there are unique $g_2,g_3 \in \prGL_2(\prC)$ with $g_2(b_i)=y_i$ and $g_3(c_i)=z_i$. Choosing additionally an invertible linear map that sends $a_0$ to $x$ shows that $t$ is in the orbit of $a_0b_0c_0 + a_0b_1c_1$. We deduce that there is exactly one $G$-orbit for $(1,2,2)$, and similarly for $(2,1,2)$, $(2,2,1)$.

Considering $(2,2,2)$, note that this multilinear rank is also attained by the tensor $w:= a_0b_0c_1 + a_0b_1c_0 + a_1b_0c_0$, because
	\begin{align*}
	\begin{pmatrix}
	0 & 1 & 1 & 0 \\ 1 & 0 & 0 & 0
	\end{pmatrix}
	\end{align*}
is a matrix presentation for all contractions of $w$. Contrary to $a_0b_0c_0 + a_1b_1c_1$, $w$ has rank three by Exercise~2'(6). Hence, we get at least two distinct orbits for $(2,2,2)$. This may be also seen as follows. Given $t = \sum_{i,j,k} t_{ijk} \: a_ib_jc_k \in V$ in coordinates, the \textit{hyperdeterminant} is
	\begin{align}\label{pr:eqHyperdet}
	Det(t) = &\left( t_{000}^2t_{111}^2 + t_{001}^2t_{110}^2 + t_{010}^2t_{101}^2 + t_{011}^2 t_{100}^2 \right) \\
	&- 2 ( t_{000}t_{001}t_{110}t_{111} + t_{000}t_{010}t_{101}t_{111} + t_{000}t_{011}t_{100}t_{111} \notag \\
	&\qquad + t_{001}t_{010}t_{101}t_{110} + t_{001}t_{011}t_{110}t_{100} + t_{010}t_{011}t_{101}t_{100} ) \notag \\
	&+ 4 \left( t_{000}t_{011}t_{101}t_{110} + t_{001}t_{010}t_{100}t_{111} \right) \, .\notag
	\end{align}
One can show, compare \cite[Section~14.1~A]{prGKZdiscriminants}, that
	\begin{align*}
	\forall \, (g_1,g_2,g_3) \in G \colon \quad Det\big( (g_1,g_2,g_3) \cdot t \big) = \det(g_1)^2 \det(g_2)^2 \det(g_3)^2 Det(t) \, .
	\end{align*}
Hence, the vanishing locus of the hyperdeterminant is $G$-invariant. Since $Det(w)=0$ but $Det(a_0b_0c_0 + a_1b_1c_1) \neq 0$, the respective orbits have to be distinct. We introduced $Det$ as it is needed for the classification of the orbits below, see Exercise 2' at the end or \cite{prOtt}.

Analogously to the case $(1,2,2)$, a tensor $t \in V$ of rank two may be written as $t = x_0y_0z_0 + x_1y_1z_1$. If $t$ has additionally multilinear rank $(2,2,2)$, then the $x_i$, the $y_i$ and the $z_i$ are linearly independent. Thus, $t$ is contained in the orbit of $a_0b_0c_0 + a_1b_1c_1$ and we conclude that there is exactly one orbit of rank two tensors with multilinear rank $(2,2,2)$. A similar argument for rank three is more cumbersome. Instead one may use geometric ideas to conclude that there are in fact \emph{exactly} two orbits for $(2,2,2)$, see Exercise~2c).

Altogether, three is the largest attained rank and there are six non-zero orbits in $V$. For stating the orbits it is more convenient to work over $\prP(V)$. Of course, the $G$-action on $V$ induces a $G$-action on $\prP(V)$ by $g \cdot [v] := [g \cdot v]$ for $[v] \in \prP(V)$. By construction, the projection $\pi \colon V \!\setminus \! \lbrace 0 \rbrace \to \prP(V)$ is $G$-equivariant. Thus, a non-zero orbit $G \cdot v$ may be recovered by $\pi^{-1} \left( G \cdot [v] \right)$. The complete characterization of the orbits of $\prP(V)$, which follows from the above discussion and Exercises~2 and 2', is given in Table~\ref{pr:tabOrbitsProj} and Table~\ref{pr:tabRepresentativesC2}. The inclusion relations between the orbit closures is illustrated in Figure~\ref{pr:figOrbitsProjC2}.

\begin{table}[ht]
  \renewcommand{\arraystretch}{1.6}
  \centering
  \begin{tabular}{|c|c|c|c|}
  \hline
  \textbf{Nr.} & \textbf{Orbit Closure} & \textbf{Orbit} & \textbf{Codim.} \\
  \hline 
  \textbf{1.} & $\prP \left( \prC^a \otimes \prC^b \otimes \prC^c \right) \cong \prP^7$ & $\prP(V) \!\setminus\! \Delta$ & 0 \\
  \hline 
  \textbf{2.} & $\Delta := \lbrace Det=0 \rbrace$ & $\Delta \!\setminus\! \big( Z_1 \cup Z_2 \cup Z_3 \big)$ & 1 \\
  \hline 
  \textbf{3.} & $Z_1 := \sigma_1 \big( \prP(\prC^a) \times \prP(\prC^b \otimes \prC^c) \big)$ & $Z_1 \!\setminus \! S$ & 3 \\ 
  \textbf{4.} & $Z_2 := \sigma_2 \big(\prP(\prC^b) \times \prP(\prC^a \otimes \prC^c) \big)$ & $Z_2 \!\setminus \! S$ & 3 \\ 
  \textbf{5.} & $Z_3 := \sigma_3 \big(\prP(\prC^a \otimes \prC^b) \times \prP(\prC^c) \big)$ & $Z_3 \!\setminus \! S$ & 3 \\ 
  \hline 
  \textbf{6.} & $S:= \sigma \left( \prP^1 \times \prP^1 \times \prP^1 \right)$ & $S$ & 4 \\ 
  \hline
  \end{tabular} 
  \caption{The $G$-orbits and their closures in $\prP(V)$ for $V = \prC^a \otimes \prC^b \otimes \prC^c$, $a=b=c=2$. $Det$ denotes the hyperdeterminant, $\sigma \colon \prP^1 \times \prP^1 \times \prP^1 \to \prP(V)$ and $\sigma_1 \colon \prP(\prC^a) \times \prP(\prC^b \otimes \prC^c) \to \prP(V)$ (similarly $\sigma_2, \sigma_3$) are Segre embeddings.}
  \label{pr:tabOrbitsProj}
\end{table}

\begin{table}[ht] 
  \renewcommand{\arraystretch}{1.5}  
  \centering
  \begin{tabular}{|c|c|c|c|c|}
  \hline
  \textbf{Nr.} & \textbf{Representative} & \textbf{Multil. Rk} & \textbf{Rk} & \textbf{Rk Decomposition} \\ 
  \hline 
  \textbf{1.} & $a_0b_0c_0 + a_1 b_1 c_1$ & (2,2,2) & 2 & unique \\ 
  \hline 
  \textbf{2.} & $a_0 b_0 c_1 + a_0 b_1 c_0 + a_1 b_0 c_0$ & (2,2,2) & 3 & infinitely many \\ 
  \hline 
  \textbf{3.} & $a_0b_0c_0 + a_0b_1c_1$ & (1,2,2) & 2 & \\ 
  \textbf{4.} & $a_0b_0c_0 + a_1b_0c_1$ & (2,1,2) & 2 & infinitely many \\ 
  \textbf{5.} & $a_0b_0c_0 + a_1b_1c_0$ & (2,2,1) & 2 & \\ 
  \hline 
  \textbf{6.} & $a_0b_0c_0$ & (1,1,1) & 1 & unique \\ 
  \hline 
  \end{tabular} 
  \caption{Representatives of the $G$-orbits of $\prP(V)$ for $V = \prC^a \otimes \prC^b \otimes \prC^c$, $a=b=c=2$. Here $\lbrace a_0, a_1 \rbrace$, $\lbrace b_0, b_1 \rbrace$ and $\lbrace c_0, c_1 \rbrace$ are bases of $\prC^a$, $\prC^b$ and $\prC^c$ respectively.}
  \label{pr:tabRepresentativesC2}
\end{table}

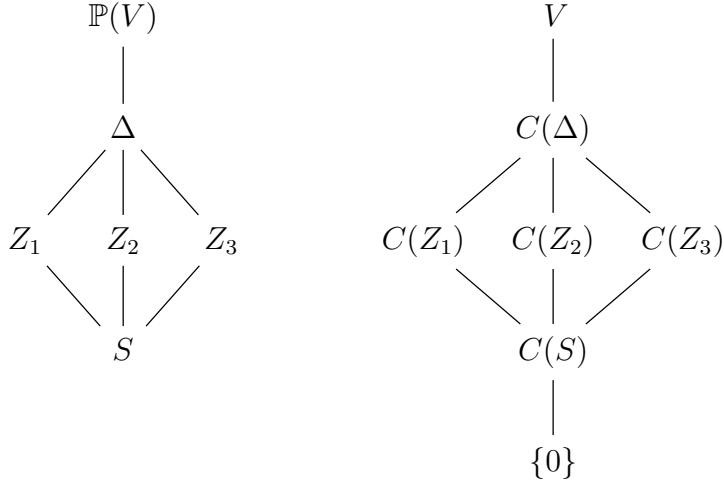
\begin{figure}[ht]
  \centering
  \begin{tikzcd}[column sep = tiny]
  & \prP(V) \ar[d, dash] & & & & & \,V \ar[d, dash] &  \\ 
  & \Delta \ar[d, dash] \ar[ld, dash] \ar[rd, dash] & & & & & C(\Delta) \ar[d, dash] \ar[ld, dash] \ar[rd, dash] & \\ 
  Z_1 \ar[rd, dash] & Z_2 \ar[d, dash] & Z_3 \ar[ld, dash] & & \quad & C(Z_1) \ar[rd, dash] & C(Z_2) \ar[d, dash] & C(Z_3) \ar[ld, dash] \\
  & S & & & & & C(S) \ar[d, dash] & \\
  & & & & & & \lbrace 0 \rbrace &
  \end{tikzcd}
  \caption{Containment graph for the orbit closures of $\prP(V)$ and for the orbit closures of $V = \prC^2 \otimes \prC^2 \otimes \prC^2$. Here, the affine cone of a Zariski-closed $X \subseteq \prP(V)$ is denoted by $C(X) \subseteq V$.}
  \label{pr:figOrbitsProjC2}
\end{figure}

We remark that the combination of tensor rank and multilinear rank suffices to classify the orbits, compare Table~\ref{pr:tabRepresentativesC2}. The orbit $\prP(V) \!\setminus\! \Delta$ is Zariski-dense in $\prP(V)$, hence a generic tensor in $V$ has rank two and multilinear rank $(2,2,2)$.  Moreover, note that the relation between orbits and their closures agrees with the fact, that orbit closures contain the orbit and possibly orbits of lower dimension.

Regarding Figure~\ref{pr:figOrbitsProjC2}, let us mention that the inclusions $S \subseteq Z_i$ are given via the product of $\prid_{\prP^1}$ with a Segre embedding $\prP^1 \times \prP^1 \hookrightarrow \prP^3$. Furthermore, one can prove that the $Z_i$ are the irreducible components of the singular locus of $\Delta$.

\medskip

We end the section by providing some geometric ideas in the style of Exercise~2 part~b). This illustrates how one can keep the six orbits of $\prP(V)$ geometrically apart. Let us start with some general thoughts on the pencils $\prP^1 \to \prP(\prC^{2 \times 2})$ induced by a tensor $t$. For this, write $t = \sum_{i,j,k} t_{ijk} \, a_i \otimes b_j \otimes c_k$ in coordinates. The first contraction of $t$ as a matrix is
	\begin{equation}\label{pr:eqPencilContraction}
	\begin{pmatrix}
	t_{000} & t_{001} & t_{010} & t_{011} \\ t_{100} & t_{101} & t_{110} & t_{111}
	\end{pmatrix}
	\end{equation}
where we equip $(\prC^a)^\vee$ with the dual basis $( \alpha_0, \alpha_1 )$ of $( a_0, a_1 )$. Using the identification $\prC^b \otimes \prC^c \cong \prC^{2 \times 2}$ the first contraction induces the pencil
	\begin{equation}\label{pr:eqPencilMap}
	\Pi \colon \prP^1 \dashrightarrow \prP \left( \prC^{2 \times 2} \right), \;(\alpha_0 : \alpha_1) \mapsto \left[ 
	\alpha_0 \begin{pmatrix} t_{000} & t_{001} \\ t_{010} & t_{011} \end{pmatrix}
	+ \alpha_1 \begin{pmatrix} t_{100} & t_{101} \\ t_{110} & t_{111} \end{pmatrix} \right] \, .
	\end{equation}
The dashed arrow indicates that $\Pi$ is in general only a rational function, since it may not be defined on certain points of $\prP^1$. 

If the first contraction of $t$ has rank one, then the rows of the matrix in \eqref{pr:eqPencilContraction} are linearly dependent. Comparing this with \eqref{pr:eqPencilMap} we conclude that the image of $\Pi$ collapses to a point in $\prP(\prC^{2 \times 2})$, because the two matrices in \eqref{pr:eqPencilMap} are linearly dependent. The latter also yields, that $\Pi$ is not everywhere defined. On the other hand, if the contraction has rank two then, for similar reasons, the image of $\Pi$ is a projective line in $\prP(\prC^{2 \times 2})$ and $\Pi$ is everywhere defined. These considerations also hold for the two other pencils coming from the second and third contraction of $t$.

Now, we examine the specific behaviour of the pencils for each orbit. To do so, we proceed in the same order as given in Tables~\ref{pr:tabOrbitsProj} and \ref{pr:tabRepresentativesC2} and choose the corresponding representatives. Let $Q \cong \prP^1 \times \prP^1$ be the Segre variety in $\prP(\prC^{2 \times 2})$. It has degree two and consists of the rank one matrices.

\medskip

\textbf{1.)} The multilinear rank is $(2,2,2)$, so all pencils are projective lines in $\prP(\prC^{2 \times 2})$. For the representative $a_0b_0c_0 + a_1b_1c_1$ all three contractions give the same pencil, namely
	\begin{align*}
	\Pi_1 \colon \prP^1 \to \prP \left( \prC^{2 \times 2} \right), \; (\alpha_0 : \alpha_1) \mapsto
	\left[ \begin{pmatrix} \alpha_0 & 0 \\ 0 & \alpha_1 \end{pmatrix}	\right] \, .
	\end{align*}
Considering the determinant $\alpha_0 \alpha_1$ shows that the image intersects $Q$ transversely in the two points $\Pi_1(1:0)$ and $\Pi_1(0:1)$. This corresponds to the orbit being disjoint from the vanishing locus of the hyperdeterminant $\Delta$.

\textbf{2.)} Again, the multilinear rank is $(2,2,2)$ and hence all pencils are projective lines. All three contractions of $a_0b_0c_1 + a_0b_1c_0 + a_1b_0c_0$ yield the pencil
	\begin{align*}
	\Pi_2 \colon \prP^1 \to \prP \left( \prC^{2 \times 2} \right), \; (\alpha_0 : \alpha_1) \mapsto
	\left[ \begin{pmatrix} \alpha_1 & \alpha_0 \\ \alpha_0 & 0 \end{pmatrix}	\right] \, .
	\end{align*}
Since the matrices in the image have determinant $- \alpha_0^2$, the projective line $\Pi_2(\prP^1)$ is tangent to $Q$ in the point $\Pi_2(0:1)$. Moreover, as indicated by the determinant this intersection has multiplicity two. All this amounts to the fact that the orbit is contained in $\Delta$.

\textbf{3.)} Here $(1,2,2)$ is the multilinear rank, so the first contraction induces a pencil that collapses to a point, while the other two pencils give rise to projective lines. Specifically, for the representative $a_0b_0c_0 + a_0b_1c_1$ the first contraction yields
	\begin{align*}
	\Pi_3 \colon \prP^1 \dashrightarrow \prP \left( \prC^{2 \times 2} \right), \; (\alpha_0 : \alpha_1) \mapsto
	\left[ \begin{pmatrix} \alpha_0 & 0 \\ 0 & \alpha_0 \end{pmatrix}	\right] \, ,
	\end{align*}
which is not defined for $(0:1)$. The image collapses to a point \emph{outside} $Q$, because the determinant is $\alpha_0^2 \neq 0$ for $\alpha_0 \neq 0$. This is in accordance with the fact that the orbit is disjoint from $\sigma_1 \big( \prP(\prC^a) \times Q^{bc} \big) \cong \prP^1 \times Q$, where $Q^{bc}$ denotes the Segre variety in $\prP \big( \prC^b \otimes \prC^c \big)$.
\\
Furthermore, the second and third contraction induce the pencil
	\begin{align*}
	\Pi_4 \colon \prP^1 \to \prP \left(\prC^{2 \times 2} \right), \; (\beta_0 : \beta_1) \mapsto
	\left[ \begin{pmatrix} \beta_0 & \beta_1 \\ 0 & 0 \end{pmatrix}	\right] \, .
	\end{align*}
Its image is completely contained \emph{in} the quadric $Q$, which is due  to the first factor $\prP^1 \cong \prP(\prC^a)$ of $\prP(\prC^a) \times \prP(\prC^b \otimes \prC^c)$.

The cases \textbf{4.)} and \textbf{5.)} are symmetric to the third orbit.

\textbf{6.)} In the case of rank one tensors all pencils collapse to a point as the multilinear rank is $(1,1,1)$. Taking $a_0b_0c_0$, all three contractions give the pencil
	\begin{align*}
	\Pi_6 \colon \prP^1 \dashrightarrow \prP \left(\prC^{2 \times 2} \right), \; (\alpha_0 : \alpha_1) \mapsto
	\left[ \begin{pmatrix} \alpha_0 & 0 \\ 0 & 0 \end{pmatrix}	\right] \, ,
	\end{align*}
which is not defined in $(0:1)$. The image lies always \emph{in} $Q$, which is due to the orbit being the Segre variety $\prP^1 \times \prP^1 \times \prP^1 \cong S \subseteq \prP(V)$.

\medskip

Finally, we conclude that the geometry of the pencils characterizes the corresponding orbit uniquely. Note that considering the representatives is enough, since, up to base change, all vectors of the respective orbit are of the same form as its representative and since the $G$-action preserves the investigated properties of the pencils. 

%%%%%%% End: Tensor Rank and Border Rank %%%%%%%

%%%%%%% Begin: Complexity of Matrix multiplication %%%%%%%

\section{Complexity of Matrix Multiplication}\label{pr:secComplexity}

In this section we will discuss the relationship between the asymptotic complexity of matrix multiplication and tensor (border) rank. Thereby, we provide a brief introduction to what has become a vast field of research, compare, e.g., the corresponding sections in the monographs \cite{prAlgComplTheoryBook}, \cite{prLandsbergBook} and \cite{prLandsbergComplexity17}. The focus will lie on early results such as \cite{prStrGaussianElimination69}, \cite{prBini2-7799} and \cite{prBini80}.

\medskip

We consider the bilinear form of matrix multiplication
\begin{align*}
\mu_n \colon \prC^{n \times n} \times \prC^{n \times n} \to \prC^{n \times n}, (A,B) \mapsto AB  \quad\text{with}\quad 
(AB)_{ik} = \sum_{j=1}^n A_{ij}B_{jk}\, .
\end{align*}
From the definition one can directly see that any of the $n^2$ entries of $AB$ can be computed with $n$ multiplications and $n-1$ additions. Hence, $\mu_n$ can be computed with $\prO(n^3)$ many operations. Nevertheless, this is just the very beginning of the story.

In his classical work \cite{prStrGaussianElimination69} from 1969 Strassen showed that $\mu_n$ can actually be computed with $\prO(n^{2.81})$ operations. At that time, the result was a big surprise! In fact, Strassen's initial goal was, ironically, to show that the standard computation is optimal with respect to the number of multiplications. This already fails in the case of $2 \times 2$ matrices: Strassen's algorithm from \cite{prStrGaussianElimination69}, presented in Example~\ref{pr:exM2} below, uses seven instead of eight multiplications for multiplying $2 \times 2$ matrices. Although the algorithm needs more additions than the standard one, it yields faster computation for \emph{large} matrices as follows. After filling up with zeros, we may assume that the two matrices have size $2^k \times 2^k$. Dividing both matrices into $2 \times 2$ block-matrices, each block of size $2^{k-1} \times 2^{k-1}$, we can apply Strassen's algorithm and proceed by recursion. The analysis in \cite{prStrGaussianElimination69} shows that, asymptotically, this method has complexity $n^{\log_2 7} \approx n^{2.81}$. We will recover this statement in Corollary~\ref{pr:corStrassen69} below.

To do so, let us introduce the standard measure for the asymptotic complexity of matrix multiplication.

\begin{defn}
The \textbf{exponent of matrix multiplication} $\omega$ is defined as
	\begin{align*}
	\omega := \inf \big\lbrace \tau \in \prR \mid \mu_n \text{ is computable in } \prO(n^\tau) \text{ arithmetic operations} \big\rbrace \, .
	\end{align*}
\end{defn}

By the above discussion we already know that $\omega \leq 2.81$. Moreover, since one has to compute $n^2$ matrix entries, it is easy to see that $2 \leq \omega$. It was a crucial observation by Strassen that $\omega$ is intimately related to the tensor rank of the following family of tensors.

\begin{defn}
Let $n\geq 1$ and let $\lbrace E_{i,j} \mid i,j=1,\ldots,n \rbrace$ be the standard basis of $\prC^{n \times n}$, i.e., the $(i,j)$ entry of $E_{i,j}$ is one and all other entries are zero. We define
	\begin{align*}
	M_{\langle n \rangle} := \sum_{i,j,k=1}^n E_{i,j} \otimes E_{j,k} \otimes E_{k,i} \; \in \prC^{n \times n} \otimes \prC^{n \times n} \otimes \prC^{n \times n} 
	\end{align*}
and call it the \textbf{tensor of matrix multiplication}.
\end{defn}

Identifying $\prC^{n \times n} \cong (\prC^{n \times n})^\vee$, the tensor $M_{\langle n \rangle}$ corresponds to the trilinear map $\varphi \colon (A,B,C) \mapsto \prtr(ABC)$. The definition is motivated by $\mu(E_{i,j}, E_{j,k}) = E_{i,k}$ and one usually switches $i$ and $k$ in the third tensor factor for convenience. Actually, this may also be justified as follows. The bilinear map $\mu$ corresponds to a tensor $(\prC^{n \times n} )^\vee \otimes ( \prC^{n \times n})^\vee \otimes \prC^{n \times n}$ and taking transposition (amounts to taking dual) in the third tensor factor gives the tensor in $(\prC^{n \times n})^\vee \otimes (\prC^{n \times n})^\vee \otimes (\prC^{n \times n})^\vee$ corresponding to $\varphi$.

\begin{remark}\label{pr:remRkMnMonotone}
For all $n \geq 2$ the tensor $M_{\langle n-1 \rangle}$ is a restriction of $M_{\langle n \rangle}$ (see Definition~\ref{pr:defnTensorRank}). To see this, let $\alpha \colon \prC^{n \times n} \to \prC^{(n-1) \times (n-1)}$ be the unique linear map that sends $E_{i,j} \in \prC^{n \times n}$ to $E_{i,j} \in \prC^{(n-1) \times (n-1)}$ if $i,j \neq n$ and to zero otherwise. Then $(\alpha \otimes \alpha \otimes \alpha)(M_{\langle n \rangle}) = M_{\langle n-1 \rangle}$ and hence $M_{\langle n-1 \rangle} \leq M_{\langle n \rangle}$. Therefore, we have 
	\begin{align*}
	\forall \, n\geq 2 \colon \quad \prrk(M_{\langle n-1 \rangle}) \leq \prrk(M_{\langle n \rangle})
	\end{align*}
by Remark~\ref{pr:remTensorRk}~d).
\end{remark}

The next theorem expresses the relation between $\omega$ and $M_{\langle n \rangle}$, that was observed by Strassen.

\begin{thm}[{implicitly in \cite{prStrGaussianElimination69}}]\label{pr:thmOmegaRk}
The exponent of the asymptotic complexity of $\prrk \left( M_{\langle n \rangle} \right)$ is equal to $\omega$, i.e.,
	\begin{align*}
	\omega = \inf \left\lbrace \tau \in \prR \mid \prrk \left( M_{\langle n \rangle} \right) \in \prO(n^\tau) \right\rbrace \, .
	\end{align*}
\end{thm}

\begin{proof}
For a detailed proof we refer to \cite[Proposition~15.1]{prAlgComplTheoryBook}.
\end{proof}

\begin{ex}\label{pr:exM2}
In the following we give a tensor decomposition, which will show $\prrk(M_{\langle 2 \rangle}) \leq 7$. In fact, Winograd \cite{prWinogradRankM2} has proven $\prrk(M_{\langle 2 \rangle}) = 7$. To improve readability we denote $E_{i,j} \in \prC^{2 \times 2}$ in the three tensor factors by $a_{ij}$, $b_{ij}$ and $c_{ij}$ respectively. For an easy comparison with Strassen's algorithm, we give a decomposition for $M'_{\langle 2 \rangle}$, the tensor obtained from acting on $M_{\langle 2 \rangle}$ by the transposition in the third tensor factor, i.e.,
	\begin{align*}
	M'_{\langle 2 \rangle}= \, &
\underbrace{a_{11}\otimes b_{11}\otimes c_{11}}_{1}
+\underbrace{ a_{12}\otimes b_{21}\otimes c_{11}}_{2}+\underbrace{a_{21}\otimes b_{11}\otimes c_{21}}_{3}+\underbrace{a_{22}\otimes b_{21}\otimes c_{21}}_{4}
\\
&
+
\underbrace{a_{11}\otimes b_{12}\otimes c_{12}}_{5}
+\underbrace{a_{12}\otimes b_{22}\otimes c_{12}}_{6}+\underbrace{a_{21}\otimes b_{12}\otimes c_{22}}_{7}+\underbrace{a_{22}\otimes b_{22}\otimes c_{22}}_{8} \, .
	\end{align*}
The corresponding decomposition for $M_{\langle 2 \rangle}$ can be obtained by applying again the transposition in the third factor. We have

\begin{align*}
M'_{\langle 2 \rangle} = \,&   \underbrace{(a_{11} + a_{22})\otimes(b_{11} + b_{22})\otimes (c_{11}+ c_{22})}_{\text{I}}
+\underbrace{(a_{21} + a_{22})\otimes b_{11}\otimes (c_{21} - c_{22})}_{\text{II}}
\\
&
+\underbrace{a_{11}\otimes (b_{12} - b_{22})\otimes (c_{12} + c_{22})}_{\text{III}} +\underbrace{a_{22} \otimes (- b_{11}+ b_{21})\otimes (c_{21} + c_{11})}_{\text{IV}}\\
&
+\underbrace{(a_{11} + a_{12}) \otimes b_{22} \otimes (- c_{11} + c_{12})}_{\text{V}}
+\underbrace{(-a_{11} + a_{21})\otimes (b_{11} + b_{12})\otimes c_{22}}_{\text{VI}}\\
&
+\underbrace{( a_{12} - a_{22})\otimes ( b_{21} + b_{22}) \otimes c_{11}}_{\text{VII}}.
\end{align*}

This decomposition is just Strassen's algorithm written as a tensor. To illustrate this, let now $A=(a_{ij}), B=(b_{ij}) \in \prC^{2 \times 2}$ and set $C = (c_{ij}) := AB$. Strassen's algorithm for multiplying $A$ and $B$ computes first
\begin{align*}
	\text{I} &= (a_{11} + a_{22})(b_{11} + b_{22}) &
	\text{II} &= (a_{21} + a_{22}) b_{11} \\
	\text{III} &= a_{11} (b_{12} - b_{22}) &
	\text{IV} &= a_{22} (- b_{11}+ b_{21}) \\
	\text{V} &= (a_{11} + a_{12})  b_{22} &
	\text{VI} &= (-a_{11} + a_{21}) (b_{11} + b_{12}) \\
	\text{VII} &= ( a_{12} - a_{22}) ( b_{21} + b_{22})
\end{align*}
and afterwards
\begin{align*}
	c_{11} &= \text{I} + \text{IV} - \text{V} + \text{VII} &
	c_{21} &= \text{II} + \text{IV} \\
	c_{12} &= \text{III} + \text{V} &
	c_{22} &= \text{I} - \text{II} + \text{III} + \text{VI} \, .
\end{align*}
The reader is encouraged to carefully compare the decomposition with the algorithm and to verify correctness.

There is also a classical decomposition of $M_{\langle 3 \rangle}$ due to Laderman \cite{prLadermanM3} showing $\prrk(M_{\langle 3 \rangle}) \leq 23$, see also \cite[Section~11.4]{prLandsbergBook}.\hfill\prExample

%\begin{equation*}
%\begin{aligned}
%M_{\langle 2 \rangle}=&
%\underbrace{\aa 11\otimes \bb 11\otimes \cc 11}_{\blue{1}}
%+\underbrace{\aa 12\otimes\bb 21\otimes \cc 11}_{\blue{2}}+\underbrace{\aa 21\otimes\bb 11\otimes\cc 21}_{\blue{3}}+\underbrace{\aa 22\otimes\bb 21\otimes\cc 21}_{\blue{4}}
%\\
%&
%+
%\underbrace{\aa 11\otimes \bb 12\otimes \cc 12}_{\blue{5}}
%+\underbrace{\aa 12\otimes\bb 22\otimes \cc 12}_{\blue{6}}+\underbrace{\aa 21\otimes\bb 12\otimes\cc 22}_{\blue{7}}+\underbrace{\aa 22\otimes\bb 22\otimes\cc 22}_{\blue{8}}\\
%\\
%\\
%  =&   \underbrace{(\aa 11 + \aa 22)\otimes(\bb 11 + \bb 22)\otimes (\cc 11+\cc 22)}_{\red{1}}
%+\underbrace{(\aa 21 + \aa 22)\otimes \bb 11\otimes (\cc 21 -\cc 22)}_{\red{2}}
%\\
%&
%+\underbrace{\aa 11\otimes (\bb12-\bb 22)\otimes (\cc 12 +\cc 22)}_{\red{3}} +\underbrace{\aa 22\otimes(-\bb 11+\bb 21)\otimes (\cc 21 +\cc 11)}_{\red{4}}\\
%&
%+\underbrace{(\aa 11+\aa 12)\otimes \bb 22\otimes (-\cc 11+\cc 12)}_{\red{5}}
%+\underbrace{(-\aa 11+\aa 21)\otimes (\bb 11 +\bb 12)\otimes \cc 22}_{\red{6}}\\
%&
%+\underbrace{(\aa 12 -\aa 22)\otimes (\bb 21 + \bb 22)\otimes \cc 11}_{\red{7}}.
%\end{aligned}
%\end{equation*} 
\end{ex}

Besides the result $\prrk(M_{\langle 2 \rangle})=7$, asking for concrete values of $\prrk(M_{\langle n \rangle})$ and $\prbrk(M_{\langle n \rangle})$ for small $n$ is challenging. Indeed, already the case $n=3$ is still open, for both rank and border rank. At least the border rank of $M_{\langle 2 \rangle}$ has been determined.

\begin{thm}[\cite{prLandsbergBorderRankM2}]
It holds that $\prbrk(M_{\langle 2 \rangle}) = 7$.
\end{thm}

A shorter proof of the latter result is contained in the recent preprint \cite[Section~5]{prCHLNewLowerBounds}.

Next, we head for exhibiting upper bounds on $\omega$ in terms of the (border) rank of $M_{\langle n \rangle}$ for \emph{fixed} $n$. In particular, we will recover Strassen's result $\omega \leq 2.81$.

\begin{defn}
Let $V_i$, $W_i$ for $i = 1,\ldots,d$ be finite dimensional $\prC$-vector spaces and choose $s \in W_1 \otimes \cdots \otimes W_d$, $t \in V_1 \otimes \cdots \otimes V_d$. Then we can form the tensor product $s \otimes t$ and we call it the \textbf{Kronecker product} of $s$ and $t$, when viewing it as a $d$-tensor in $V'_1 \otimes \cdots \otimes V'_d$ with $V'_i := W_i \otimes V_i$.
\end{defn}

\begin{remark}\label{pr:remKroneckerProduct}
Note that it is important how one views the tensor product $s \otimes t$. For example, the rank of $s \otimes t$ as the Kronecker product, i.e., as a $d$-tensor, may be strictly smaller than its rank as a $2d$-tensor in $V_1 \otimes \cdots \otimes V_d \otimes W_1 \otimes \cdots \otimes W_d$. To see this, consider $U,V,W,X = \prC^2$ and corresponding basis vectors $u_i,v_i,w_i,x_i$ for $i=1,2$. Then
	\begin{align*}
	\left( \sum_{i=1}^2 u_i \otimes v_i \right) \otimes (w_1 \otimes x_1) = u_1 \otimes v_1 \otimes w_1 \otimes x_1 + u_2 \otimes v_2 \otimes w_1 \otimes x_1
	\end{align*}
has rank one in $(U \otimes V)\otimes(W \otimes X)$, but not in $U \otimes V \otimes W \otimes X$ as it does not have multilinear rank $(1,1,1,1)$.
\end{remark}

The Kronecker product has the following useful properties.

\begin{lem}\label{pr:lemKroneckerProduct}
Let $t_1 \in V_1 \otimes \cdots \otimes V_d$, $t_2 \in W_1 \otimes \cdots \otimes W_d$ and consider their Kronecker product $t_1 \otimes t_2$. Then
	\begin{abclist}
	\item $\prrk(t_1 \otimes t_2) \leq \prrk(t_1) \prrk(t_2)$.
	
	\item Let $d=3$. If $t_1 \unlhd_{q_1} \langle r_1 \rangle$ and $t_2 \unlhd_{q_2} \langle r_2 \rangle$, then $t_1 \otimes t_2 \unlhd_{q_1 + q_2} \langle r_1 r_2 \rangle$.
	\end{abclist}
\end{lem}

\begin{proof}
For part a) choose rank decompositions
	\begin{align*}
	t_1 = \sum_{\rho=1}^{\prrk(t_1)} v_{1,\rho} \otimes \cdots \otimes v_{d,\rho} \quad \text{and} \quad
	t_2 = \sum_{\theta=1}^{\prrk(t_2)} w_{1,\theta} \otimes \cdots \otimes w_{d,\theta}
	\end{align*}
with $v_{i,\rho} \in V_i$, $w_{i,\theta} \in W_i$ for $i=1,\ldots,d$. Tensoring these decompositions yields
	\begin{align*}
	t_1 \otimes t_2 = \sum_{\rho=1}^{\prrk(t_1)} \sum_{\theta=1}^{\prrk(t_2)} \left( v_{1,\rho} \otimes w_{1,\theta} \right) \otimes \cdots \otimes \left( v_{d,\rho} \otimes w_{d,\theta} \right)
	\end{align*}
and hence $\prrk(t_1 \otimes t_2) \leq \prrk(t_1) \prrk(t_2)$.

To prove part b) choose representations as in equation~\eqref{pr:eqDegeneration}, i.e.,
	\begin{align*}
	\pre^{q_1} t_1 + \pre^{q_1+1}s_1(\pre) &= \sum_{\rho=1}^{r_1} v_{1,\rho}(\pre) \otimes v_{2,\rho}(\pre) \otimes v_{3,\rho}(\pre) \\
	\pre^{q_2} t_2 + \pre^{q_2+1}s_2(\pre) &= \sum_{\theta=1}^{r_2} w_{1,\theta}(\pre) \otimes w_{2,\theta}(\pre) \otimes w_{3,\theta}(\pre) \, ,
	\end{align*}
where $v_{i,\rho}(\pre) \in \prC[\pre] \otimes_{\prC} V_i$, $w_{i,\theta}(\pre) \in \prC[\pre] \otimes_{\prC} W_i$ and $s_1 \in \prC[\pre] \otimes_{\prC} (V_1 \otimes V_2 \otimes V_3)$, $s_2 \in \prC[\pre] \otimes_{\prC} (W_1 \otimes W_2 \otimes W_3)$. Forming the tensor product gives a representation
	\begin{align*}
	\pre^{q_1 + q_2} (t_1 \otimes t_2) + \pre^{q_1 + q_2 + 1} s(\pre) =
	\sum_{\rho=1}^{r_1} \sum_{\theta=1}^{r_2} \bigotimes_{i=1}^3 \big(\, v_{i,\rho}(\pre) \otimes w_{i,\theta}(\pre)\,\big)
	\end{align*}
with $s(\pre) = [s_1(\pre) \otimes t_2] + [t_1 \otimes s_2(\pre)] + \pre [s_1(\pre) \otimes s_2(\pre)]$. Therefore, we obtain the claim $t_1 \otimes t_2 \unlhd_{q_1 + q_2} \langle r_1 r_2 \rangle$.
\end{proof}

An immediate consequence of part a) is the following corollary.

\begin{cor}\label{pr:corRkKroneckerProd}
Let $t \in V_1 \otimes \cdots \otimes V_d$. For any $k \geq 1$ the Kronecker product $t^{\otimes k}$ satisfies $\prrk \left( t^{\otimes k} \right) \leq \prrk(t)^k$.
\end{cor}

The tensor of matrix multiplication admits the following nice property.

\begin{lem}\label{pr:lemTensorProdOfMn}
If $n,m \in \prN$, then $M_{\langle n \rangle} \otimes M_{\langle m \rangle} = M_{\langle nm \rangle}$ for the Kronecker product of $M_{\langle n \rangle}$ and $M_{\langle m \rangle}$.
\end{lem}

\begin{proof}
The proof is left as an exercise, see Exercise 4 below.
\end{proof}

With these properties of the Kronecker product we are now able to present two ways of finding upper bounds on $\omega$, namely Propositions~\ref{pr:propBoundOnOmega} and \ref{pr:propBoundBrkOmega}.

\begin{prop}\label{pr:propBoundOnOmega}
If $\prrk \left( M_{\langle n \rangle} \right) \leq r$, then $\omega \leq \log_n r$ respectively $n^\omega \leq r$.
\end{prop}

\begin{proof}
Corollary~\ref{pr:corRkKroneckerProd} for $t = M_{\langle n \rangle}$ combined with Lemma~\ref{pr:lemTensorProdOfMn} shows
	\begin{equation}\label{pr:eqBoundOnOmega}
	\prrk \left( M_{\langle n^k \rangle} \right) = \prrk \left( M_{\langle n \rangle}^{\otimes k} \right) \leq \prrk \left( M_{\langle n \rangle} \right)^k \leq r^k = n^{(\log_n r)k}
	\end{equation}
for all $k \geq 1$. Now, for an arbitrary $m \geq 1$, equation \eqref{pr:eqBoundOnOmega} with $k := \lceil \log_n m \rceil$ together with $\prrk(M_{\langle n_1 \rangle}) \leq \prrk(M_{\langle n_2 \rangle})$ for $n_1 \leq n_2$ (see Remark~\ref{pr:remRkMnMonotone})
gives
	\begin{align*}
	\prrk \left( M_{\langle m \rangle} \right) \leq \prrk \left( M_{\langle n^k \rangle} \right) \leq n^{(\log_n r) \lceil \log_n m \rceil} \leq n^{(\log_n r)(1 + \log_n m)} = r \, m^{\log_n r} \, .
	\end{align*}
We conclude $\prrk( M_{\langle m \rangle} ) \in \prO(m^{\log_n r})$ and Theorem~\ref{pr:thmOmegaRk} yields $\omega \leq \log_n r$.
\end{proof}

With the bound $\prrk(M_{\langle 2 \rangle}) \leq 7$ from Example~\ref{pr:exM2} we recover Strassen's classical result from 1969.

\begin{cor}[\cite{prStrGaussianElimination69}]\label{pr:corStrassen69}
It holds that $\omega \leq \log_2 7 \leq 2.81$.
\end{cor}

Remember from Section~\ref{pr:secRank} that the set of all tensors with border rank at most $r$ is Zariski-closed, while this fails in general for the rank. Thus, the notion of \emph{border rank} is better suited than \emph{rank} for geometric methods. Hence, it is good news that there are analogues of Proposition~\ref{pr:propBoundOnOmega} and Theorem~\ref{pr:thmOmegaRk} for border rank.

\begin{prop}[\cite{prBini80}]\label{pr:propBoundBrkOmega}
For all $n \geq 1$ one has $n^\omega \leq \prbrk\left(M_{\langle n \rangle}\right)$. In particular, a bound $\prbrk\left(M_{\langle n \rangle}\right) \leq r$ implies $\omega \leq \log_n r$.
\end{prop}

\begin{proof}
Fix $n$. Then there is some $q \in \prN$ such that $M_{\langle n \rangle} \unlhd_q \prbrk(M_{\langle n \rangle})$. Part b) of Lemma~\ref{pr:lemKroneckerProduct} in combination with Lemma~\ref{pr:lemTensorProdOfMn} implies
	\begin{align*}
	M_{\langle n^k \rangle} \unlhd_{kq} \left\langle \, \prbrk \left( M_{\langle n \rangle} \right)^k \, \right\rangle
	\end{align*}
for all integers $k \geq 1$. Thus, by Lemma~\ref{pr:lemApproxToExact}
	\begin{align*}
	\prrk\left( M_{\langle n^k \rangle} \right) \leq (kq + 1)^2 \prbrk \left( M_{\langle n \rangle} \right)^k
	\end{align*}
and applying Proposition~\ref{pr:propBoundOnOmega} to the latter gives
	\begin{align*}
	n^{k \omega} \leq (kq + 1)^2 \prbrk \left( M_{\langle n \rangle} \right)^k \, , \quad \text{hence } \quad
	n^{\omega} \leq \sqrt[k]{(kq + 1)^2} \,\prbrk \left( M_{\langle n \rangle} \right)
	\end{align*}
for any $k \geq 1$. Letting $k$ tend to infinity we conclude $n^\omega \leq \prbrk( M_{\langle n \rangle} )$. The second part of the statement follows from the monotonicity of the logarithm.
\end{proof}

In \cite{prBini2-7799} {Bini et al.\ implicitly} observed that $M_{\langle 12 \rangle} \unlhd_4 1000$ and therefore $\prbrk(M_{\langle 12 \rangle}) \leq 1000$, although the notion of border rank was not yet defined at this time. They used the former to deduce the following bound on $\omega$, which is for us a direct consequence of Proposition~\ref{pr:propBoundBrkOmega}.

\begin{cor}[\cite{prBini2-7799}]
We have $\omega \leq \log_{12} 1000 < 2.78$.
\end{cor}

Furthermore, Proposition~\ref{pr:propBoundBrkOmega} enables us to prove the counterpart of Theorem~\ref{pr:thmOmegaRk} for the \emph{border} rank of $M_{\langle n \rangle}$.

\begin{thm}[\cite{prBini80}]\label{pr:thmOmegaBrk}
The exponent of the asymptotic complexity of $\prbrk \left( M_{\langle n \rangle} \right)$ is equal to the exponent of matrix multiplication, i.e.,
	\begin{align*}
	\omega = \inf \left\lbrace \tau \in \prR \mid \prbrk \left( M_{ \langle n \rangle} \right) \in \prO(n^\tau) \right\rbrace \, .
	\end{align*}
\end{thm}

\begin{proof}
Set $\alpha := \inf \lbrace \tau \in \prR \mid \prbrk ( M_{ \langle n \rangle} ) \in \prO(n^\tau) \rbrace$. The inequality $\omega \geq \alpha$ follows from Theorem~\ref{pr:thmOmegaRk} and $\prrk(M_{\langle n \rangle}) \geq \prbrk(M_{\langle n \rangle})$. Conversely, $n^\omega \leq \prbrk(M_{\langle n \rangle})$ for all $n \geq 1$ (Proposition~\ref{pr:propBoundBrkOmega}) yields $\omega \leq \alpha$.
\end{proof}

Summarizing, we have seen how to express and to upper bound the exponent $\omega$ of matrix multiplication via the (border) rank of $M_{\langle n \rangle}$. Although the bound $\omega < 2.78$ due to $\prbrk(M_{\langle 12 \rangle}) \leq 1000$ is not a huge improvement of Strassen's result $\omega \leq 2.81$, the border rank turned out to be more powerful. Namely, Sch\"onhage \cite{prSchoenhageTauTheorem} generalized Proposition~\ref{pr:propBoundBrkOmega} to his Asymptotic Sum Inequality (also called $\tau$-theorem) and was thus able to prove $\omega < 2.55$. Two further landmarks are Strassen's laser method ($\omega < 2.48$) from \cite{prStrRelBilinearComplexity87} and its probabilistic refinement ($\omega < 2.376$) by Coppersmith and Winograd \cite{prCoppersmithWinograd2-38}. During the last three decades the progress almost stopped and the current state of the art is $\omega < 2.374$ due to Le Gall \cite{prLeGall14}. For additional historical details the interested reader may consult \cite[Section~15.13]{prAlgComplTheoryBook} or the survey \cite{prStrEuropeanCongress}.

On the other hand, there is an interest for finding concrete lower bounds for the complexity of matrix multiplication. For this, general lower bounds on $\prbrk( M_{\langle n \rangle} )$ have been studied:
\begin{enumerate}
	\item Strassen \cite{prStrassen83}: $\frac{3}{2}n^2 \leq \prbrk(M_{\langle n \rangle})$
	\item Lickteig \cite{prLickteig84}: $\frac{3}{2}n^2 + \frac{1}{2} n - 1 \leq \prbrk(M_{\langle n \rangle})$
	\item B\"urgisser-Ikenmeyer\footnote{Compared to the other mentioned works, the article \cite{prBI13} provides a completely different method of proof. The lower bound is obtained by exhibiting so-called obstructions, a representation-theoretic notion coming from Geometric Complexity Theory.} \cite{prBI13}: $\frac{3}{2} n^2 - 2\leq \prbrk(M_{\langle n \rangle})$
	\item Landsberg-Ottaviani \cite{prLanOttav}: $2n^2 - n \leq \prbrk(M_{\langle n \rangle})$
	\item Landsberg-Michalek \cite{prLandsbergMichalek}: $2n^2 - \log n -1 \leq \prbrk(M_{\langle n \rangle})$
	\end{enumerate}
However, all these lower bounds cannot improve the trivial bound $2 \leq \omega$ for the asymptotic complexity! Actually, as a special case of Conjecture~\ref{pr:conjConciseTight} below, there is the following astonishing conjecture.

\begin{conj}\label{pr:conjOmega2}
$\omega = 2$.
\end{conj}

\emph{That is, asymptotically matrix multiplication is conjectured to be nearly as easy as matrix addition!}

Although $\omega = 2$ would be very astounding, it seems to be believed by some experts, and it is considered to be equiprobable to $\omega > 2$ by other experts. Nevertheless, let us point out the discrepancy between the \emph{theoretic} measure $\omega$ and real world implementations. 
Already for upper bounds on $\omega$ the hidden constant in the $\prO$-notation may be so large, that a corresponding algorithm would just be impractical. To stress this, we explain the meaning of Conjecture~\ref{pr:conjOmega2}. For any $\delta > 0$ there is a constant $c(\delta) > 0$ such that computing $\mu_n$ needs (at most) $c(\delta)n^{2 + \delta}$ arithmetic operations. Not only that $c(\delta)$ may be huge, it may also tend \textit{very fast} to infinity for $\delta \to 0$.

In \cite[Section~4]{prLandsbergComplexity17} this is circumvented by introducing the notion $\omega_{prac,k}$, which does \emph{not} contain a hidden constant. There, $\omega_{prac,k}$ is used to study the complexity of multiplying matrices of size at most $k \times k$ and for further details we refer to \cite{prLandsbergComplexity17}.

When it comes to algorithms used in practice, there are in fact only few known to beat Strassen's algorithm from 1969. Regarding concrete implementations and papers studying practical issues we refer to \cite{prBBfastParallelMatrixMult}, \cite{prBodrato}, \cite{prANadaptiveStrassen}, \cite{prHJJTTimplementation}, \cite{prSmirnov13} and the references therein.

\medskip

Concluding this section, the following important problems remain open.
\begin{numlist}
 \item Determine the exponent of matrix multiplication. In particular, is $\omega = 2$?
 \item Compute $\prrk \left( M_{\langle n \rangle} \right)$ and $\prbrk \left( M_{\langle n \rangle} \right)$ for some small $n \geq 3$.
\end{numlist}

%%%%%%% End: Complexity of Matrix multiplication %%%%%%%

%%%%%%% Begin: Symmetric Tensors and Symmetric Rank %%%%%%%

\section{Symmetric Tensors and Symmetric Rank}\label{pr:secSymmetric}

In the case $V_1 = \ldots = V_d$ we can study the symmetry of tensors in $V_1 \otimes \cdots \otimes V_d$. Let $V$ be an $n$-dimensional $\prC$-vector space and denote the symmetric group of $\lbrace 1,\ldots,d \rbrace$ by $\mathfrak{S}_d$. Of course, $\mathfrak{S}_d$ acts linearly on $V^{\otimes d}$ by permuting the tensor factors, i.e., on decomposable tensors the action is given by
	\begin{align*}
	\sigma \cdot (v_1 \otimes \cdots \otimes v_d) := v_{\sigma(1)} \otimes \cdots \otimes v_{\sigma(d)}
	\end{align*}
for $\sigma \in \mathfrak{S}_d$ and $v_1,\ldots,v_d \in V$.

\begin{defn}
A tensor $t \in V^{\otimes d}$ is called \textbf{symmetric}, if $\sigma \cdot t = t$ for all $\sigma \in \mathfrak{S}_d$, i.e., $t$ is $\mathfrak{S}_d$-invariant. The subspace of symmetric tensors of $V^{\otimes d}$ is denoted by $\prSym^d(V)$.
\end{defn}

\begin{remark}\label{pr:remSymmetricTensor} Let $e_1,\ldots,e_n$ be a basis of $V$ and write
	\begin{align*}
	t = \sum_{i_1,\ldots,i_d} t_{i_1,\ldots,i_d} \; e_{i_1} \otimes \cdots \otimes e_{i_d} \in V^{\otimes d} \, .
	\end{align*}
Then
	\begin{abclist}
	\item $t \in \prSym^d(V)$ if and only if $t_{i_1,\ldots,i_d} = t_{\sigma(i_1), \ldots, \sigma(i_d)}$ for all $\sigma \in \mathfrak{S}_d$ and all $i_1,\ldots,i_d \in \lbrace 1,\ldots, n \rbrace$.
	
	\item Part a) allows us to identify $\prSym^d(V)$ with the space of  homogeneous polynomials of degree $d$ on $V^{\vee}$. Often we will do this implicitly.
	\item The \textbf{\textit{symmetrization}} of $t$,
	\begin{align*}
	t^{\prsym} := \sum_{i_1,\ldots,i_d} \left( \frac{1}{d!} \sum_{\sigma \in \mathfrak{S}_d} t_{\sigma(i_1),\ldots,\sigma(i_d)} \right) e_{i_1} \otimes \cdots \otimes e_{i_d} \, ,
	\end{align*}
	is a symmetric tensor and $t \in \prSym^d(V)$ if and only if $t=t^{\prsym}$.
	\end{abclist}
\end{remark}

\begin{defn}
Let $t \in \prSym^d(V)$. The \textbf{symmetric rank} or \textbf{Waring rank} of $t$ is
	\begin{align*}
	\prsrk(t) := \min \left\lbrace r \;\Big\vert\; t = \sum_{i=1}^r (l_i)^d \, , \; l_i \in \prSym^1 V \right\rbrace \, .
	\end{align*}
Similarly to the border rank, we define the \textbf{border symmetric rank} of $t$ as
	\begin{align*}
	\prbsrk(t) := \min \left\lbrace r \; \Big\vert \; t \text{ is a limit of } t_i \in \prSym^d(V) \text{ with } \prsrk(t_i) = r \right\rbrace \, .
	\end{align*}
\end{defn}

The (border) symmetric rank is an important notion with many applications, see, e.g., \cite[Part 3]{prLandsbergBook}. Regarding the exponent $\omega$ of matrix multiplication, Theorem~\ref{pr:thmChilo} below is an analogue of Theorems~\ref{pr:thmOmegaRk} and \ref{pr:thmOmegaBrk} in the ``symmetric world''. Namely, we can characterize $\omega$ also by the (border) symmetric rank, when considering the \emph{symmetrization} of the tensors $M_{\langle n \rangle}$, $n \geq 1$.

\begin{remark}\label{pr:remSrk} Let $t \in \prSym^d(V)$.
	\begin{abclist}
	\item It holds that $\prbsrk(t) \leq \prsrk(t)$, $\prrk(t) \leq \prsrk(t)$ and $\prbrk(t) \leq \prbsrk(t)$.
	
	\item Let $\nu_d \colon \prP(V) \to \prP(\prSym^d (V)), \; [v] \mapsto [v^d]$ be the Veronese embedding. Then $\prsrk(t)=1$ if and only if $t \neq 0$ is in the affine cone over the Veronese variety $\nu_d(\prP(V))$. 
	
	\item If $W$ is a $\prC$-vector space containing $V$, then $\prSym^d(V) \subseteq \prSym^d(W)$ and the symmetric rank of $t$ is independent of viewing $t$ as an element in $\prSym^d(V)$ or in $\prSym^d(W)$. To see this, we consider $t$ as a homogeneous polynomial of degree~$d$, i.e., $t \in R_d$ for $R := \prC[x_1,\ldots,x_n]$, and let $y$ be an additional variable. Then any decomposition of $t$ into a sum of powers of linear forms in $R$ is also such a decomposition in $R[y]$. \\
	Conversely,  if there are linear forms $l_1,\ldots,l_r \in R[y]_1$ such that
	\begin{align*}
	t(x_1,\ldots,x_n) = \sum_{i=1}^r \big( \, l_i(x_1,\ldots,x_n,y) \, \big)^d \, ,
	\end{align*}
	then setting $y=0$ gives on the right hand side a sum of $r$ powers of linear forms in $R$.
	\end{abclist}
\end{remark}

\begin{ex}
For $V = \prC^2$ consider the homogeneous polynomial $t = 3x^2y$. It corresponds to the symmetric tensor $x \otimes x \otimes y + x \otimes y \otimes x + y \otimes x \otimes x$, which we already encountered in Example~\ref{pr:exWstate}. We already know $\prrk(t)=3$ and $\prbrk(t)=2$. Actually, the computation for $\prbrk(t)=2$ also shows $\prbsrk(t)=2$ as
	\begin{align*}
	w = \lim_{\prEps \to 0} \; \frac{1}{\prEps} \left[ (x+ \prEps y)^3 - x^3 \right] \, .
	\end{align*}
Moreover, we have the decomposition
\begin{align*}
t = 3x^2y = \frac{1}{2} (x+y)^3 - \frac{1}{2}(x-y)^3 -y^3
\end{align*}
and therefore $\prsrk(t) \leq 3$ and with $\prrk(t)=3$ we conclude $\prsrk(t) = 3$. The latter can also be seen with Theorem~\ref{pr:ThmSrkMonomial} below.\hfill\prExample
\end{ex}

\begin{ex}
Another example for a strict inequality $\prbsrk(t) < \prsrk(t)$ is the following. Consider $V= \prC^2$ and $t := x^3 + 3 x^2y$. By Exercise 5, $t$ is not the sum of two cubes, hence $\prsrk(t) \geq 3$. On the other hand, $t$ is the limit of polynomials, that are sums of two cubes, namely
	\begin{align*}
	t &= \lim_{\prEps \to 0} \; \frac{1}{\prEps} \left[ (\prEps-1)x^3 + (x+\prEps y)^3 \right] \\
	&= \lim_{\prEps \to 0} \; \frac{1}{\prEps} \left[ \prEps x^3 + 3 \prEps x^2y + 3 \prEps^2 xy^2 + \prEps^3y^3 \right]
	\end{align*}
Therefore $\prbsrk(t)=2$.\hfill\prExample
\end{ex}

\begin{remark}
The inequality $\prrk(t) \leq \prsrk(t)$ for $t \in \prSym^d (V)$ may be strict as well. An astounding example with $d=3$ and $n=800$ variables, where $\prrk(t) \leq 903$ and $\prsrk(t) = 904$, is due to Shitov \cite{prShitov18}.
\end{remark}

%%%%%%% End: Symmetric Tensors and Symmetric Rank %%%%%%%

%%%%%%% Begin: Apolarity Theory %%%%%%%

\section{Apolarity Theory}\label{pr:secApolarity}

This section introduces Apolarity Theory, which already dates back to works of Sylvester from 1851, see \cite{prSylvester1}, \cite{prSylvester2} and \cite{prSylvester3}. It can be used to compute the symmetric rank of some homogeneous polynomials as we shall see in Section~\ref{pr:secExamples}. The main tool for this is the Apolarity Lemma, which we state in a ``reduced version'' in Theorem~\ref{pr:thmRedApolarity} and in its ``scheme version'' in Theorem~\ref{pr:thmSchemeApolarity}. For further details on Apolarity Theory the reader is referred to the literature, e.g., \cite{prIKbook} and \cite{prRSvarietiesOfSums}.

Let $V$ be an $n+1$ dimensional $\prC$-vector space and denote its dual by $V^\vee$. We also consider the symmetric algebras $S= \prC[x_0,\ldots, x_n] := \prSym (V)$ and $T= \prC[\partial_0,\ldots,\partial_n] := \prSym (V^{\vee})$. As the labeling of the variables suggests we let $T$ act linearly on $S$ by formal differentiation. This action will be indicated by a dot, e.g., $g \cdot f$ for $f \in S$ and $g \in T$. We note that $g_1 \cdot (g_2 \cdot f) = (g_1 g_2) \cdot f$ for all $f \in S$ and all $g_1,g_2 \in T$. Moreover, given multi-indices $\alpha = (\alpha_0,\ldots,\alpha_n) \in (\prZ_{\geq 0})^{n+1}$ and $\beta = (\beta_0,\ldots,\beta_n) \in (\prZ_{\geq 0})^{n+1}$ we introduce the shortcuts
	\begin{align*}
	\partial^\alpha := \partial_0^{\alpha_0} \partial_1^{\alpha_1} \cdots \partial_n^{\alpha_n}
	\quad \text{ and } \quad x^\beta := x_0^{\beta_0} x_1^{\beta_1} \cdots x_n^{\beta_n} 
	\end{align*}
as well as
	\begin{align*}
	&\vert \alpha \vert := \sum_{i=0}^n \alpha_i \, , \quad
	\alpha! := \prod_{i=0}^n \alpha_i ! \quad \text{ and } \quad
	{{d}\choose{\alpha}} := \frac{d!}{\alpha !} = \frac{d!}{\alpha_0 ! \cdots \alpha_n !} \, ,
	\end{align*}
where in the latter $d = \vert \alpha \vert$.

\begin{lem}\label{pr:lemDiff}
Let $\alpha$ and $\beta$ be multi-indices with $\vert \alpha \vert = \vert \beta \vert$, then
	\begin{align*}
	\partial^\alpha \cdot x^\beta = \begin{cases}
    \alpha!      & \quad \text{if } \alpha = \beta\\
    0  & \quad \text{if } \alpha \neq \beta
  \end{cases}
	\end{align*}
\end{lem}

\begin{proof}
Clearly, $\partial^\alpha \cdot x^\alpha = \alpha!$ by the rules of formal differentiation. On the other hand, if $\alpha \neq \beta$ then $\vert \alpha \vert = \vert \beta \vert$ yields some $j$ such that $\alpha_j > \beta_j$. The latter implies $\partial^\alpha \cdot x^\beta = 0$.
\end{proof}

Let us point out a direct consequence of Lemma~\ref{pr:lemDiff}. The bilinear map
	\begin{equation}\label{pr:eqDualPair}
	\prSym^d(V^\vee) \times \prSym^d(V) \to \prC, \quad (g,f) \mapsto g \cdot f
	\end{equation}
is a dual pairing. It gives an isomorphism $\prSym^d(V^\vee) \cong \big( \prSym^d(V) \big)^\vee$ under which $(\partial^{\alpha})_{\vert \alpha \vert = d}$ becomes the dual basis of $\big( (\alpha!)^{-1} x^{\alpha}\big)_{\vert \alpha \vert = d}$.

\begin{lem}\label{pr:lemLinearForm}
Let $g \in T_d = \prSym^d (V^\vee)$ and let $l = \sum_{i=0}^n c_i x_i \in S_1 = \prSym^1 (V) = V$, where $c_i \in \prC$. Then $g \cdot l^d = d! \, g(c_0,c_1,\ldots,c_n)$.
\end{lem}

\begin{proof}
The multinomial theorem gives
	\begin{align*}
	l^d = \sum_{\vert \beta \vert =d} {{d}\choose{\beta}} c^\beta x^\beta \quad \text{ and we can write } \quad g = \sum_{\vert \alpha \vert = d} g_\alpha \partial^\alpha
	\end{align*}
with $g_\alpha \in \prC$. Applying Lemma~\ref{pr:lemDiff} we conclude
	\begin{align*}
	g \cdot l^d = \sum_{\vert \alpha \vert = d} g_\alpha c^\alpha \alpha! {{d} \choose{\alpha}} = \sum_{\vert \alpha \vert = d} g_\alpha c^\alpha d! = d! \, g(c_0,c_1,\ldots,c_n) \, ,
	\end{align*}
which is the claim.
\end{proof}

\begin{defn}
The \textbf{annihilator} or \textbf{apolar ideal} of $f \in \prSym^d(V)$ is the homogeneous ideal
	\begin{align*}
	\prAnn(f) := f^\perp := \big\lbrace g \in \prSym(V^\vee) \mid g \cdot f = 0 \big\rbrace
	\end{align*}
of $\prSym(V^\vee)$. Its $d$-th homogeneous part $(f^\perp)_d$ is called the \textbf{socle} of $f^\perp$. Moreover, as $f^\perp$ is homogeneous, we can consider the graded ring
	\begin{align*}
	A_f := \prSym(V^\vee) / (f^\perp) = \bigoplus_{e=0}^\infty \, \prSym^e(V^\vee) / (f^\perp)_e \, ,
	\end{align*}
which is called the \textbf{apolar ring} of $f$.
\end{defn}

The notation $A_f$ is quite common in the literature. To avoid confusion, let us point out that the apolar ring is \emph{not} related to localization at all.

\begin{remark} Let $f \in S_d = \prSym^d(V)$, $f \neq 0$.
	\begin{abclist}
	\item The socle $(f^\perp)_d$ has codimension one in the $\prC$-vector space $\prSym^d(V^\vee)$.
	\item If $k > d$, then $(f^\perp)_k = \prSym^k(V^\vee)$.
	\item By part b) the graded $\prC$-algebra $A_f$ is Artinian, because
	$(A_f)_k = 0$ for all $k > d$ and $(A_f)_e$ is finite dimensional for all $e \leq d$.
	\end{abclist}
\end{remark}

The following proposition will be needed to prove the Apolarity Lemma, Theorem~\ref{pr:thmSchemeApolarity}.

\begin{prop}\label{pr:propSocle}
Let $f \in \prSym^d (V)$. The apolar ideal $f^\perp$ is determined by its socle $(f^\perp)_d$, namely for all $e \le d$
	\begin{align*}
	(f^\perp)_e = \left[ (f^\perp)_d : \prm^{d-e} \right]_e :=
	\left\lbrace g \in T_e \mid \forall h \in \prm^{d-e} \colon\, (gh) \cdot f = 0  \right\rbrace ,
	\end{align*}
where $\prm := (\partial_0, \ldots, \partial_n)$ is the irrelevant ideal of $T$.
\end{prop}

\begin{proof}
Since $f^\perp$ is an ideal, the inclusion $(f^\perp)_e \subseteq [ (f^\perp)_d : \prm^{d-e} ]_e$ follows immediately. Conversely, for $g \in [ (f^\perp)_d : \prm^{d-e}]_e$ we have $(g \, \partial^\alpha) \cdot f = \partial^\alpha \cdot (g \cdot f) = 0$ for all multi-indices $\alpha$ with $\vert \alpha \vert = d-e$. Together with Lemma~\ref{pr:lemDiff} this implies that all coefficients of $g \cdot f \in S_{d-e}$ are zero. Thus $g \cdot f = 0$, i.e., $g \in (f^\perp)_e$.
\end{proof}

The next proposition is equivalent to saying that $A_f$ is a \emph{Gorenstein} Artinian ring.

\begin{prop}
Let $f \in \prSym^d(V)$ and $e \in \lbrace0,1,\ldots,d \rbrace$. The multiplication
	\begin{align*}
	(A_f)_e \times (A_f)_{d-e} \to (A_f)_d \cong \prC
	\end{align*}
is a perfect pairing. In particular, $\dim_{\prC} (A_f)_e = \dim_{\prC} (A_f)_{d-e}$.
\end{prop}

\begin{proof}
We write $[g]$ for the equivalence class of $g \in T$ in $A_f = T / (f^\perp)$. By symmetry, it is enough to show that the pairing is non-degenerate in one component. Let $[t] \in (A_f)_e$ with $[tu] = 0$ in $(A_f)_d$ for all $[u] \in (A_f)_{d-e}$. In particular, $tu \in (f^\perp)_d$ for all $u \in \prm^{d-e} \subseteq T_{d-e}$, i.e., $t \in [(f^\perp)_d : \prm^{d-e}]_e$. Finally, Proposition~\ref{pr:propSocle} implies $t \in (f^\perp)_e$, i.e., $[t] = 0$ in $(A_f)_e$.
\end{proof}

\begin{ex}
Let $f = x^\alpha \in S$ be a monomial for some multi-index $\alpha$. Then
	\begin{align*}
	f^\perp = \left( \partial_0^{\alpha_0 + 1}, \ldots, \partial_n^{\alpha_n + 1} \right) \;\; \text{ and } \;\;
	A_f = \prC[\partial_0,\ldots,\partial_n] / (\partial_0^{\alpha_0 + 1}, \ldots, \partial_n^{\alpha_n + 1}) \, .
	\end{align*}
Since $A_f$ has Krull dimension zero and is generated by homogeneous elements of degree one, it holds that $\deg A_f = \dim_{\prC} A_f = (\alpha_0 + 1)(\alpha_1 + 1) \cdots (\alpha_n + 1)$.\hfill\prExample
\end{ex}

Next, we turn to the main result of this section, the Apolarity Lemma. It was a smart idea due to Sylvester to link the differential operators killing $f$ with the decompositions of $f$ as sum of powers of linear forms. 

To formulate the statement, recall that we identify $\prSym^d(V^\vee)$ with the space of homogeneous polynomials on $V$, compare Remark~\ref{pr:remSymmetricTensor}~b).  Therefore, we view $\prSym(V^\vee)$ as the homogeneous coordinate ring of $\prP(V)$. For a closed subscheme $Z \subseteq \prP(V)$ let $I_Z \subseteq \prSym(V^\vee)$ denote the unique saturated ideal corresponding to $Z$; it is the vanishing ideal of $Z$.

Now, we state the reduced version of the Apolarity Lema. It is instructive to also reformulate it in a down-to-earth way for binary forms.

\begin{thm}[Apolarity Lemma, reduced version]\label{pr:thmRedApolarity} \ \\
Let $Z = \lbrace [l_1],\ldots,[l_k] \rbrace \subseteq \prP(V)$ be a subscheme of closed reduced points with vanishing ideal $I_Z \subseteq \prSym(V^\vee)$. Then, for $f \in \prSym^d (V)$,
	\begin{align*}
	I_Z \subseteq f^\perp \quad \Leftrightarrow \quad
	\exists \, c_i \in \prC \colon \; f  = \sum_{i=1}^k c_i l_i^d \, .
	\end{align*}
\end{thm}

\begin{thm}[Reduced Apolarity Lemma for binary forms]\label{pr:thmRedApolarityBinary} \ \\
Let $f \in \prC[x,y]_d$ and pick distinct $(\alpha_i : \beta_i) \in \prP^1$ for $i=1,\ldots,k$. Then
	\begin{align*}
	\prod_{i=1}^k (\beta_i \partial_x - \alpha_i \partial_y) \cdot f = 0 \quad \Leftrightarrow \quad \exists \, c_i \in \prC \colon \; f = \sum_{i=1}^k c_i (\alpha_i x + \beta_i y)^d \, .
	\end{align*}
\end{thm}

Since $\prC$ is algebraically closed, we may write $c_i l_i^d = (l'_i)^d$, where $l'_i$ a linear form with $[l_i] = [l'_i]$,
by taking a $d$-th root of $c_i$. Thus, the reduced Apolarity Lemma characterizes the symmetric rank of $f$ as the smallest $k$ such that there is a closed reduced subscheme $Z \subseteq \prP(V)$ consisting of $k$ distinct points and satisfying $I_Z \subseteq f^\perp$. This will be used in Section~\ref{pr:secExamples} for computing the symmetric rank of certain polynomials. We omit a proof of the reduced version as it is a special case of the scheme-theoretic version in Theorem~\ref{pr:thmSchemeApolarity} below.

To formulate that theorem, we define the \emph{projective linear span} of a closed subscheme $X \subseteq \prP^N$, denoted by $\langle X \rangle$, to be the smallest projective linear subspace of $\prP^N$, which contains $X$ as a sub\emph{scheme}. The scheme $\langle X \rangle$ is the vanishing locus of $(I_X)_1$.

For illustrative reasons we state the Apolarity Lemma for binary forms, Theorem~\ref{pr:thmSchemeApolarityBinary}, back to back with the general version.  The  formulation of Theorem~\ref{pr:thmSchemeApolarityBinary} can also be found in  \cite[Lemma~1.31]{prIKbook}. 

\begin{thm}[Apolarity Lemma, scheme version]\label{pr:thmSchemeApolarity} \ \\
Let $Z \subseteq \prP(V)$ be a closed zero-dimensional subscheme with vanishing ideal $I_Z \subseteq \prSym(V^\vee)$ and let $\nu_d \colon \prP(V) \to \prP ( \prSym^d(V) ), \, [l] \mapsto [l^d]$ be the Veronese embedding. Then,  for $f \in \prSym^d(V) \backslash \{0\}$,
	\begin{align*}
	I_Z \subseteq f^\perp \quad \Leftrightarrow \quad [f] \in \langle \nu_d(Z) \rangle \, .
	\end{align*}
\end{thm}

\begin{thm}[Apolarity Lemma for binary forms]\label{pr:thmSchemeApolarityBinary} \ \\
Let $f \in \prC[x,y]_d$, pick distinct $(\alpha_i : \beta_i) \in \prP^1$ and integers $1 \leq m_i \leq d$ for $i=1,\ldots,k$. Then
	\begin{align*}
	&\prod_{i=1}^k (\beta_i \partial_x - \alpha_i \partial_y)^{m_i} \cdot f = 0 \\ \Leftrightarrow \quad &\exists \, c_i(x,y) \in \prC[x,y]_{m_i -1} \colon \; f =  \sum_{i=1}^k c_i(x,y) (\alpha_i x + \beta_i y)^{d-m_i +1} \, .
	\end{align*}
\end{thm}

Before proving Theorem~\ref{pr:thmSchemeApolarity}, let us further stress in which sense the scheme version of the Apolarity Lemma generalizes the reduced version.  For binary forms, the reader is encouraged to compare Theorem~\ref{pr:thmRedApolarityBinary} with Theorem~\ref{pr:thmSchemeApolarityBinary}, thereby noting how the scheme version ``sees multiplicities''.

In general, if $X \subseteq \prP^N$ is a closed reduced subscheme, then $\langle X \rangle$ is given by the usual projective linear span of the closed points of $X$. Thus, if $Z$ (and hence $\nu_d(Z)$) is reduced in Theorem~\ref{pr:thmSchemeApolarity} we obtain Theorem~\ref{pr:thmRedApolarity}.

In contrast, for non-reduced $X$ one may have $\langle X_{\operatorname{red}} \rangle \varsubsetneq \langle X \rangle$ as ``fat points need more space''. For example, if $Y \subseteq \prP^2$ is the closed reduced subscheme consisting of the points $(1:0:0)$ and $(0:1:0)$, then $\langle Y \rangle = \prProj(\,\prC[x_0,x_1,x_2]/(x_2)\,) \cong \prP^1$. Equipping the point $(1:0:0)$  with the multiplicity two, scheme structure coming from the ideal
$(x_1,x_2^2)$,  and leaving $(0:1:0)$ untouched, we get a scheme $X$ with
	\begin{align*}
		I_X = \big[ (x_1, x_2^2)(x_0,x_2) \big]^{\text{sat}} = \big[ (x_0 x_1, x_1 x_2, x_0x_2^2, x_2^3) \big]^{\text{sat}}
		= (x_0 x_1, x_1 x_2, x_2^2) ,
	\end{align*}
where $[\cdot]^{\text{sat}}$ denotes saturation.
Thus, $X \nsubseteq \langle Y \rangle$ as sub\emph{schemes} of $\prP^2$ since $x_2 \notin I_X$. Hence, we necessarily get $\langle X \rangle = \prP^2$, which may also be seen via $(I_X)_1=0$.

\begin{proof}[Proof of Theorem~\ref{pr:thmSchemeApolarity}.] %TODO
First, we note that the coordinate ring of $\prP ( \prSym^d(V) )$ is $R := \prSym \big( (\prSym^d(V))^\vee \big)$. In particular, $R_1 = (\prSym^d(V))^\vee$.  It is a property of the Veronese embedding that linear forms vanishing on $\nu_d(Z)$ correspond to homogeneous forms of degree $d$ vanishing on $Z$, i.e., 
	\begin{equation}\label{pr:eqPropVeronese}
	R_1 \cap I_{\nu_d(Z)} = (\prSym^d(V))^\vee \cap I_{\nu_d(Z)} \cong \prSym^d(V^\vee) \cap I_Z \, .
	\end{equation}
We can make this explicit as follows.

Viewing $g \in \prSym^d(V^\vee)$ as a linear form on $\prSym^d(V)$ via the dual pairing from Equation~\eqref{pr:eqDualPair}, we have that $\langle g, f'\rangle = g \cdot f'$, where $\langle g, f' \rangle$ denotes the function value of $g$ at $f' \in \prSym^d(V)$. In particular, $\langle g,f \rangle = g \cdot f$. This identification is allowed, i.e., it respects \eqref{pr:eqPropVeronese}, because for $l = \sum_i c_i x_i \in V$ we have $\langle g, l^d \rangle = g \cdot l^d = 0$ if and only if $g(l) = g(c_0,\ldots,c_n) = 0$, by Lemma~\ref{pr:lemLinearForm}.

Now, $[f] \in \langle \nu_d(Z) \rangle$ if and only if every linear form in $R_1 = (\prSym^d(V))^\vee$, that vanishes on $\nu_d(Z)$ also vanishes on $[f]$. This condition is equivalent to $[f]$ being annihilated by the space of linear forms on
$\prP ( \prSym^d(V) )$ that vanish on $\nu_d(Z)$; this space coincides with the degree $d$ forms on $\prP(V)$ that vanish on $Z$, that is $(I_Z)_d$. We have proved that $[f] \in \langle \nu_d(Z) \rangle$ if and only if $f\in(I_Z)_d^\perp$, which
is equivalent to  $(I_Z)_d \subseteq (f^\perp)_d$.

We end the proof by showing that $(I_Z)_d \subseteq (f^\perp)_d$ is equivalent to $I_Z \subseteq f^\perp$. Clearly, the latter implies the former. For the converse recall that for all $e \!  > \! d$, $(f^\perp)_e = \prSym^e (V^\vee)$ and hence $(I_Z)_e \subseteq (f^\perp)_e$. Using $(I_Z)_d \subseteq (f^\perp)_d$ and then Proposition~\ref{pr:propSocle} yields for all $1 \leq e < d$
	\begin{align*}
	(I_Z)_e \subseteq \left[ (I_Z)_d : \prm^{d-e} \right]_e \subseteq \left[ (f^\perp)_d : \prm^{d-e} \right]_e = (f^\perp)_e \, .
	\end{align*}
Altogether, we have $I_Z \subseteq f^\perp$ as desired.
\end{proof}

The scheme-theoretic version of the Apolarity Lemma is used to characterize  a notion, which was of increasing importance during the last years. Namely, the \emph{cactus rank} of a symmetric tensor $f \in \prSym^d(V)$ is the least length of any zero-dimensional subscheme $Z \subseteq \prP^n$ with $I_Z \subseteq f^\perp$. Actually, cactus rank already appeared as \emph{scheme length} in \cite{prIKbook} and a generalization of the above definition is due to \cite{prBBSecantVeronese14}.
 
Finally, let us point out the recent development of \emph{border apolarity} from \cite{prBB20apolarity}, which combines apolarity theory with the so-called border substitution method.  This promising tool was already crucially used in \cite{prCHLNewLowerBounds} to prove new bounds on the border rank of small matrix multiplication tensors.

%%%%%%% Begin: Examples of Symmetric Rank %%%%%%%

\section{Examples of Symmetric Rank}\label{pr:secExamples}

In the following we stick to the notation of Section~\ref{pr:secApolarity} and will use Apolarity Theory to study the symmetric rank in certain examples. In particular, Theorem~\ref{pr:ThmSrkMonomial} gives a formula for the symmetric rank of any monomial. Moreover, Theorem~\ref{pr:thmChilo} at the end will describe $\omega$ in terms of the (border) symmetric rank of the symmetrization of $M_{\langle n \rangle}$. We start with investigating the symmetric rank of a binary form.

\begin{ex}\label{pr:exSrkBinary}
Let $f \in \prC[x,y]_d$ be a binary form of degree $d \geq 1$. We are going to show that $\prsrk(t) \leq d$. For this, let
	\begin{align*}
	\nu_d \colon \prP^1 \to \prP^{d} \cong \prP(\prC[x,y]_d) \, , \quad (\alpha:\beta) \mapsto 
	\left( \alpha^d : \alpha^{d-1}\beta : \ldots : \beta^d \right) 
	\end{align*}
be the Veronese embedding. Its image $C := \nu_d(\prP^1)$ is the rational normal curve, which has degree $d$. Therefore, a general hyperplane $H \cong \prP^{d-1}$ of $\prP^d$ cuts $C$ in $d$ distinct, reduced points. Using the Vandermonde determinant one can deduce that these $d$ distinct points are linearly independent and hence span the hyperplane $H$. Taking $H$ to be a general hyperplane, which contains $[f] \in \prP^d$, we conclude that $f$ is a $\prC$-linear combination of $d$ many powers of linear forms. This shows $\prsrk(f) \leq d$.

The upper bound is tight, because the binary form $xy^{d-1}$ has symmetric rank $d$. Although this is a special case of Theorem~\ref{pr:ThmSrkMonomial} below, we give a direct argument to illustrate the Apolarity Lemma for binary forms, see Theorem~\ref{pr:thmRedApolarityBinary}. Of course, $\prsrk(x) = 1$ and for all $(\alpha,\beta) \in \prC^2 \!\setminus\! \lbrace 0 \rbrace$ one has $xy \neq (\alpha x + \beta y)^2$. Therefore, we may assume $d \geq 3$. Clearly, $\prsrk(xy^{d-1}) > 1$. Since $(xy^{d-1})^\perp = (\partial_x^2, \partial_y^d)$ we have
	\begin{equation}\label{pr:eqSrkBinary}
	\left( xy^{d-1} \right)^\perp_e = \left\lbrace \partial_x^2 g \; \big\vert \; g \in \prC[\partial_x,\partial_y]_{e-2} \right\rbrace
	\end{equation}
for all $2\leq e < d$. If there are closed, reduced points $(\alpha_i : \beta_i) \in \prP^1$ with
	\begin{align*}
	\prod_{i=1}^e (\beta_i \partial_x - \alpha_i \partial_y) \in 
	\left( xy^{d-1} \right)^\perp_e
	\end{align*}
and $2 \leq e < d$, then $\big\vert \lbrace i \mid \alpha_i = 0 \rbrace \big\vert \geq 2$ by equation \eqref{pr:eqSrkBinary}. Thus, the $(\alpha_i : \beta_i)$ are not all pairwise distinct. As a consequence of Theorem~\ref{pr:thmRedApolarityBinary} the form $xy^{d-1}$ cannot have symmetric rank $2 \leq e<d$, hence $\prsrk(xy^{d-1})=d$.

A detailed study of the (border) symmetric rank of binary forms can be found in \cite{prCSrankBinaryForm}.\hfill\prExample
\end{ex}

Before we prove Theorem~\ref{pr:ThmSrkMonomial} we need to recall some facts about the Hilbert function and Hilbert polynomial. For details we refer to \cite[I.7]{prHartshorne}. Let $M$ be a finitely generated graded $T$-module (recall $T = \prC[\partial_0, \ldots, \partial_n]$). Then the \textbf{\textit{Hilbert function}} of $M$ is defined as
	\begin{align*}
	HF(M, \cdot) \colon \prZ \to \prZ, \quad d \mapsto \dim_{\prC} M_d \, .
	\end{align*}
There exists an integer $N \geq 0$ and a polynomial $HP(M, x) \in \prC[x]$ such that $HF(M, d) = HP(M, d)$ for all $d \geq N$. The polynomial $HP(M,x)$ is called the \textbf{\textit{Hilbert polynomial}} of $M$.

\begin{lem}\label{pr:lemHilbertFunction}
Let $I$ be a homogeneous ideal of $T = \prC[\partial_0, \ldots, \partial_n]$ such that $T/I$ has Krull dimension one. Then the Hilbert polynomial $HP(T/I,x)$ equals some integer constant $s$. Furthermore:
	\begin{abclist}
	\item If $I$ is a radical ideal, then it is the vanishing ideal of $s$ distinct closed, reduced points in $\prP^n$.
	\item If $g \in T_1$ is a linear form, which is not a zero divisor in $T/I$, then
		\begin{align*}
		HF(T/I,d) = \sum_{i=0}^d HF \big(\, T/ ( I + (g) ), i \,\big)
		\end{align*}
	for all $d \geq 0$.	 
	\end{abclist}
\end{lem}

\begin{proof}
For all statements except b) we refer to \cite[page 52]{prHartshorne}. For part b), just note that the multiplication with $g$ yields for any $d \in \prZ$
	\[
	\begin{tikzcd}[column sep = scriptsize]
	0 \ar[r] & (T/I)_{d-1} \ar[r, "\cdot g"] & (T/I)_d \ar[r] & \left[ T/ \big(I+ (g) \big) \right]_d \ar[r] & 0 \, ,
	\end{tikzcd}
	\]
a short exact sequence of $\prC$-vector spaces.
\end{proof}

Equipped with this lemma we are now able to compute the symmetric rank of a monomial $x_0^{\alpha_0} \cdots x_n^{\alpha_n}$. Of course, we may assume $\alpha_0 \leq \ldots \leq \alpha_n$ after reordering the variables. Moreover, note that by part c) of Remark~\ref{pr:remSrk} there is no loss in generality assuming $\alpha_0 \geq 1$ as well.

\begin{thm}[{\cite[Proposition~3.1]{prCCG12}}]\label{pr:ThmSrkMonomial}
Let $1 \leq \alpha_0 \leq \alpha_1 \leq \ldots \leq \alpha_n$. Then
	\begin{align*}
	\prsrk \left( x_0^{\alpha_0} \cdots x_n^{\alpha_n} \right) = \prod_{i=1}^n (\alpha_i + 1) \, .
	\end{align*}
\end{thm}

\begin{proof}
This proof completely follows \cite{prCCG12}. For $n=0$ we have $\prsrk(x_0^{\alpha_0}) = 1$ as desired. Therefore, we can assume $n \geq 1$. Set $f := x^\alpha$ and remember that the annihilator of $f$ is
	\begin{align*}
	f^{\perp} = \left( \partial_0^{\alpha_0 + 1}, \ldots, \partial_n^{\alpha_n + 1} \right) \, .
	\end{align*}
Since $\alpha_0 = \min_i \alpha_i$ we have $\partial_0^{\alpha_i + 1} \in (f)^\perp$ for all $i=1,\ldots,n$. Hence, the homogeneous ideal
	\begin{align*}
	J := \left( \partial_0^{\alpha_1 + 1} - \partial_1^{\alpha_1 + 1},
	\partial_0^{\alpha_2 + 1} - \partial_2^{\alpha_2 + 1}, \ldots,
	\partial_0^{\alpha_n + 1} - \partial_n^{\alpha_n + 1} \right)
	\end{align*}
is contained in $f^\perp$. Moreover, $J$ is the vanishing ideal of the reduced closed subscheme
	\begin{align*}
	Z = \left\lbrace \left( 1: \xi_1^{k_1} : \xi_2^{k_2} : \ldots : \xi_n^{k_n} \right) \; \Big\vert \; 0 \leq k_i \leq \alpha_i \text{ for } i=1, \ldots,n  \right\rbrace \subseteq \prP^n \, ,
	\end{align*}
where $\xi_i$ is a primitive $(\alpha_i + 1)$-th root of unity. The cardinality of $Z$ is
	\begin{align*}
	r := \prod_{i=1}^n (\alpha_i + 1) 
	\end{align*}
and hence the reduced Apolarity Lemma (Theorem~\ref{pr:thmRedApolarity}) gives $\prsrk(f) \leq r$.

On the other hand, by Theorem~\ref{pr:thmRedApolarity} there exists an ideal $I \subseteq f^\perp$ such that $I$ is the vanishing ideal of $s:=\prsrk(f)$ distinct, closed reduced points in $\prP^n$. We are left to prove $s \geq r$. To do so, consider $I' := (I : \partial_0)$. Since $I$ is a radical ideal, also $I'$ is radical as the following computation shows
	\begin{align*}
	I' \subseteq \sqrt{I'} = \sqrt{(I : \partial_0)} \subseteq \left( \sqrt{I} : \partial_0 \right) = (I : \partial_0) = I' \, .
	\end{align*}
Moreover, since $I$ is radical we have
	\begin{align*}
	V(I: \partial_0) = \overline{V(I) \setminus V(\partial_0)} \qquad &\text{in } \mathbb{A}^{n+1} \\
	\text{respectively } \quad V_{\prP^n}(I: \partial_0) = \overline{V_{\prP^n}(I) \setminus V_{\prP^n}(\partial_0)} \qquad &\text{in } \prP^n \, .
	\end{align*}
But $\alpha_0 \geq 1$ implies $\partial_0 \notin f^\perp$ and therefore $\partial_0 \notin I$. In particular, not all points of $V_{\prP^n}(I)$ are contained in the $\partial_0$-plane, hence $V_{\prP^n}(I:\partial_0) \neq \emptyset$. Altogether, $I' = (I : \partial_0)$ is the vanishing ideal of $s' \leq s$ reduced closed points in $\prP^n$ with $s' > 0$. Hence, for $d \gg 0$ we obtain $HF(T/I', d) = s'$ by Lemma~\ref{pr:lemHilbertFunction} part a).

We finish the proof by establishing  the inequality $s' \geq r$. First note that $I \subseteq f^\perp$ yields
	\begin{equation}\label{pr:eqSrkMonomial}
	I' + (\partial_0) \subseteq J' := (f^\perp : \partial_0) + (\partial_0) = \left( \partial_0, \partial_1^{\alpha_1 + 1}, \ldots, \partial_n^{\alpha_n + 1} \right) \, .
	\end{equation}
Let us show that the linear form $\partial_0$ is not a zero divisor in $T/I'$. To prove this, we use that $I$ is radical together with $\partial_0 \notin I$ to deduce $\partial_0 \partial_0 \notin I$. Therefore, $\partial_0 \notin I' = (I: \partial_0)$ and so $\partial_0 \neq 0$ in $T/I'$. Next, assume $g \in T$ with $g \partial_0 = 0$ in $T/I'$, i.e., $g \partial_0 \in I'$ and hence $g \partial_0^2 \in I$. But then $g^2 \partial_0^2 \in I$ as well and $I$ being radical gives $g \partial_0 \in I$, so $g \in I'$. The latter means $g = 0$ in $T/I'$ as desired.

Thus, we can apply Lemma~\ref{pr:lemHilbertFunction} part b) to conclude that for $d \gg 0$
	\begin{align*}
	s' = HF(T/I', d) \overset{\ref{pr:lemHilbertFunction} \text{ b)}}{=} &\sum_{i=0}^d HF \big( T/(I' + (\partial_0)), i \big) \\ \overset{\eqref{pr:eqSrkMonomial}}{\geq} \, &\sum_{i=0}^d HF \left( T/J', i \right) = \prod_{i=1}^n (\alpha_i + 1) = r.
	\end{align*}
The last equality holds, because $J'$ is a complete intersection ideal, compare equation \eqref{pr:eqSrkMonomial}. This ends the proof as $s \geq s'$.
\end{proof}

Note that the first part of the proof of Theorem~\ref{pr:ThmSrkMonomial} provides a way of computing a Waring rank decomposition of a monomial, compare \cite[Proposition~4.3]{prCCG12}. The coefficients of such a decomposition are made explicit in \cite[Section~2]{prBBT13}.

\begin{ex}
By Theorem~\ref{pr:ThmSrkMonomial} the monomial $xyz$ has symmetric rank $2 \cdot 2 =4$. A symmetric rank decomposition is given by
	\begin{align*}
	xyz = \frac{1}{24} \left[ (x+y+z)^3 - (x+y-z)^3 - (x-y+z)^3 + (x-y-z)^3 \right] \, ,
	\end{align*}
which is actually a decomposition in the style of \cite[Proposition~4.3]{prCCG12}. Moreover, the decomposition shows that any $(uvw) := (u \otimes v \otimes w)^{\prsym} \in \prSym^3(\prC^n)$ with $u,v,w \in \prC^n$ satisfies
	\begin{equation}\label{pr:eqSrkUVW}
	\prsrk(uvw) \leq 4 \, .
	\end{equation}
Actually, $\prsrk(uvw) < 4$ for linearly dependent $u,v,w$. See \cite{Fischer94} for the monomial
$x_1\ldots x_n$.\hfill\prExample
\end{ex}

\begin{remark} A shortcut to prove Theorem~\ref{pr:ThmSrkMonomial} was provided by Buczy\'{n}ski and Teitler in
\cite[Example~9]{prBucTei}.
\end{remark}

In contrast, computing the border symmetric rank of a monomial is still an open problem.

\begin{conj}[{\cite[Conjecture~1.1]{prOedingBorderRk16}}]
For $0 \leq \alpha_0 \leq \ldots \leq \alpha_n$ one has
	\begin{align*}
	\prbsrk \left( x_0^{\alpha_0} \cdots x_n^{\alpha_n} \right) = \prod_{i=0}^{n-1} (\alpha_i + 1) \, .
	\end{align*}
\end{conj}

Finally, we give a characterization of the exponent $\omega$ of matrix multiplication via (border) symmetric rank.

\begin{defn}
Let $f_n \in \prSym^3(\prC^{n \times n})$ be the symmetrization of $M_{\langle n \rangle}$, compare part~c) of Remark~\ref{pr:remSymmetricTensor}. Note that $f_n$ is the tensor, which corresponds to the cubic form $A \mapsto \prtr(A^3)$ on $\prC^{n \times n}$.
\end{defn}

\begin{thm}[{\cite[part of Theorem~1.1]{prChilo}}]\label{pr:thmChilo}
It holds that
	\begin{align*}
	\omega &= \inf \big\lbrace \tau \in \prR \mid \prsrk(f_n) \in \prO(n^\tau) \big\rbrace
	= \inf \big\lbrace \tau \in \prR \mid \prbsrk(f_n) \in \prO(n^\tau) \big\rbrace \\
	&= \inf \big\lbrace \tau \in \prR \mid \prrk(f_n) \in \prO(n^\tau) \big\rbrace
	= \inf \big\lbrace \tau \in \prR \mid \prbrk(f_n) \in \prO(n^\tau) \big\rbrace \, .
	\end{align*}
\end{thm}

\begin{proof}
We completely follow the proof in \cite{prChilo}. First, given any tensor $t \in \prC^N \otimes \prC^N \otimes \prC^N$ we claim that $\prsrk(t^{\prsym}) \leq 4 \prrk(t)$. To see this, choose a rank decomposition $t = \sum_i u_i \otimes v_i \otimes w_i$. Then the symmetrization of $t$ is $\sum_i (u_i v_i w_i)$, where $(u_i v_i w_i) := (u_i \otimes v_i \otimes w_i)^{\prsym}$. Thus, $\prsrk(t^{\prsym}) \leq 4 \prrk(t)$ by  equation~\eqref{pr:eqSrkUVW}. We conclude
	\begin{align*}
	\prrk(f_n) \leq \prsrk(f_n) \leq 4 \prrk \left( M_{\langle n \rangle} \right)
	\end{align*}
and Theorem~\ref{pr:thmOmegaRk} implies that

	\begin{equation}\label{pr:eqChilo}
	\inf \big\lbrace \tau \in \prR \mid \prrk(f_n) \in \prO(n^\tau) \big\rbrace
	\leq \inf \big\lbrace \tau \in \prR \mid \prsrk(f_n) \in \prO(n^\tau) \big\rbrace \leq \omega \, .
	\end{equation}
	
Conversely, for matrices $A,B,C \in \prC^{n \times n}$ consider
	\begin{align*}
	X := \begin{pmatrix}
	0 & 0 & A \\ 
	C & 0 & 0 \\ 
	0 & B & 0
	\end{pmatrix} \in \prC^{3n \times 3n}
	\end{align*}
and compute
	\begin{align*}
	X^3 = \begin{pmatrix}
	ABC & 0 & 0 \\ 
	0 & CAB & 0 \\ 
	0 & 0 & BCA
	\end{pmatrix} \in \prC^{3n \times 3n} \, .
	\end{align*}
Therefore, $\prtr(X^3) = 3 \prtr(ABC)$ and this captures the fact that $3 M_{\langle n \rangle}$ (hence also $M_{\langle n \rangle}$) is a restriction of the tensor $f_{3n} = (M_{\langle 3n \rangle})^{\prsym}$. In the language of Section~\ref{pr:secRank} this can be expressed as follows. We denote by $\lbrace E_{i,j} \rbrace_{i,j}$ and $\lbrace X_{k,l} \rbrace_{k,l}$ the standard bases of $\prC^{n \times n}$ and $\prC^{3n \times 3n}$ respectively. Following the block description of the matrix~$X$ above we set respectively
	\begin{align*}
	\varphi_1 &\colon \prC^{3n \times 3n} \to \prC^{n \times n} , \; X_{k,l} \mapsto 
	\begin{cases}
    E_{k,l-2n}      & , 1 \leq k \leq n, \, 2n+1 \leq l \leq 3n \\
    0  & , \text{otherwise}
    \end{cases}
    \\
    \varphi_2 &\colon \prC^{3n \times 3n} \to \prC^{n \times n} , \; X_{k,l} \mapsto 
	\begin{cases}
    E_{k-2n,l-n}      & , 2n+1 \leq k \leq 3n, \, n+1 \leq l \leq 2n \\
    0 & , \text{otherwise}
    \end{cases}
	\end{align*}
and similarly $\varphi_3 \colon \prC^{3n \times 3n} \to \prC^{n \times n}$. Then $(\varphi_1 \otimes \varphi_2 \otimes \varphi_3) (f_{3n}) = 3M_{\langle n \rangle}$ and so $M_{\langle n \rangle} \leq f_{3n}$. Hence, by Remark~\ref{pr:remTensorRk}~d) we obtain $\prrk(M_{\langle n \rangle}) \leq \prrk(f_{3n})$. The latter inequality combined with Theorem~\ref{pr:thmOmegaRk} yields
	\begin{align*}
	\omega \leq \inf \big\lbrace \tau \in \prR \mid \prrk(f_n) \in \prO(n^\tau) \big\rbrace \, .
	\end{align*}
Together with the inequalities in \eqref{pr:eqChilo} this implies the claim for the symmetric rank and the rank of $f_n$.

For the border rank statements notice that $\prbsrk(u_i v_i w_i) \leq \prsrk(u_i v_i w_i) \leq 4$ and taking limits in the argument from the beginning yields $\prbsrk(t^{\prsym}) \leq 4 \prbrk(t)$. Therefore, we have
	\begin{align*}
	\prbrk(f_n) \leq \prbsrk(f_n) \leq 4 \prbrk \left( M_{\langle n \rangle} \right) \, .
	\end{align*}
Again by taking limits, the argument via the $\varphi_m$ shows $\prbrk(M_{\langle n \rangle}) \leq \prbrk(f_{3n})$. Finally, the claims for border symmetric rank and border rank of $f_n$ follow from Theorem~\ref{pr:thmOmegaBrk}.
\end{proof}
%For all results of \cite[Theorem~1.1]{prChilo} we refer the reader to the corresponding paper.

In \cite{prChilo} two advantages of the symmetric approach for determining $\omega$ are briefly discussed. First, it allows to use the vast knowledge from algebraic geometry on cubic hypersurfaces. Second, in comparison to the matrix multiplication tensor $M_{\langle n \rangle}$, the polynomial $f_n$ is defined on a much smaller space, thereby allowing more computational experiments. The latter may be used to gain additional data for devising new conjectures. Moreover, the symmetric setting enables the usage of Apolarity Theory from Section~\ref{pr:secApolarity}. In particular, we have seen that the Apolarity Lemma is quite handy for providing bounds on the symmetric rank of a form or even for exact computation. Altogether, it may be easier to determine $\omega$ via $f_n$ rather than $M_{\langle n \rangle}$.

Still, the symmetric approach has the drawback of missing an analogue of $M_{\langle n \rangle} \otimes M_{\langle m \rangle} = M_{\langle nm \rangle}$, compare Lemma~\ref{pr:lemTensorProdOfMn}. Indeed, it seems that for the symmetric setting there are no counterparts of Proposition~\ref{pr:propBoundOnOmega} or \ref{pr:propBoundBrkOmega}.

%%%%%%% End: Examples of Symmetric Rank %%%%%%%

%%%%%%% Begin: Asymptotic Rank %%%%%%%

\section{Asymptotic Rank}\label{pr:secAsymptotic}

This last section catches a glimpse of yet another topic related to the exponent~$\omega$ of matrix multiplication. Namely, we will introduce the asymptotic rank of a tensor $f \in V_1 \otimes V_2 \otimes V_3$.

For this, we consider the Kronecker products $f^{\otimes k}$, i.e., $ f^{\otimes k} \in V'_1 \otimes V'_2 \otimes V'_3$, where $V'_i = V_i^{\otimes k}$ and $k \geq 1$. The next definition is due to Gartenberg \cite{prGartenberg} and measures the asymptotic behaviour of $\prrk(f^{\otimes k})$.

\begin{defn}\label{pr:defnAsymRk}
For $f \in V_1 \otimes V_2 \otimes V_3$, the \textbf{asymptotic rank} of $f$ is defined as
	\begin{align*}
	\praR(f) := \lim_{k \to \infty} \big(  \prrk_3 \left( f^{\otimes k} \right) \big)^{\frac{1}{k}} \, ,
	\end{align*}
where $f^{\otimes k}$ is the Kronecker product and the rank is computed looking at $f^{\otimes k}$ as a $3$-dimensional tensor.
\end{defn}

Let us explain why the asymptotic rank is well-defined. First remember that the rank of a Kronecker product is submultiplicative, compare Lemma~\ref{pr:lemKroneckerProduct} part~a). Thus, we can apply Lemma~\ref{pr:lemFekete} below, known as \emph{Fekete's Lemma}, to the sequence $a_k := \log \prrk(f^{\otimes k})$ and deduce that $\praR(f)$ exists.

\begin{lem}[\cite{prFekete}]\label{pr:lemFekete}
Let $a_k$ be a sequence of non-negative real numbers such that $a_{n+m} \leq a_n + a_m$ for all $n,m$. Then $a_k/k$ has a limit and
	\begin{align*}
	\lim_{k \to \infty} \frac{a_k}{k} = \inf_{k\in \prN} \; \frac{a_k}{k} \, .
	\end{align*}
\end{lem}

\begin{proof}
The proof is left as an exercise, compare Exercise~6.
\end{proof}

\begin{remark}
Let $f \in V_1 \otimes V_2 \otimes V_3$.
	\begin{abclist}
	\item $\praR(f) \leq \prbrk(f) \leq \prrk(f)$
	\item In the definition of $\praR(f)$ one may replace $\prrk(f)$ by $\prbrk(f)$, the proof is similar to the
	one of Theorem \ref{pr:thmAsymptRankMn} below.
	\end{abclist}
\end{remark} 

Considering the asymptotic rank of $M_{\langle n \rangle}$ we recover the exponent $\omega$ of matrix multiplication.

\begin{thm}[{\cite[1.3]{prStrAsymptSpectrum88}}]\label{pr:thmAsymptRankMn}
For all $n \geq 2$, we have $\praR \left(M_{\langle n \rangle} \right) = n^\omega$.
\end{thm}

\begin{proof}
We follow \cite{prStrAsymptSpectrum88}. For the whole proof fix $n \geq 2$. Recall that by Theorem~\ref{pr:thmOmegaRk} we have
	\begin{equation}\label{pr:eqAsymptRankMn}
	\omega = \inf \left\lbrace \tau \in \prR \mid \prrk \left( M_{\langle m \rangle} \right) \in \prO(m^\tau) \right\rbrace \, .
	\end{equation}
Thus, for any $\tau \in \prR$ with $\omega < \tau$ we obtain with Lemma~\ref{pr:lemTensorProdOfMn} that
	\begin{align*}
	\prrk \left( M_{\langle n \rangle}^{\otimes k} \right) = \prrk \left( M_{\langle n^k \rangle} \right) \in \prO \Big( \big( n^k \big)^\tau \Big) \; \text{ for } k \to \infty \, ,
	\end{align*}
i.e., $\prrk \left( M_{\langle n \rangle}^{\otimes k} \right) \leq c n^{k \tau}$ for some constant $c > 0$ and $k \gg 0$. We deduce
	\begin{align*}
	\praR \left( M_{\langle n \rangle} \right) \leq \lim_{k \to \infty} \left( cn^{k \tau} \right)^{\frac{1}{k}} = n^\tau
	\end{align*}
and choosing $\omega < \tau$ arbitrarily close gives $\praR(M_{\langle n \rangle}) \leq n^\omega$. We finish the proof by showing that $\omega$ is the smallest real number with the latter property.

For this, let $\rho \in \prR$ be such that $\praR(M_{\langle n \rangle}) \leq n^\rho$. For $\prEps := 0.1$ and $k \gg 0$ we have
	\begin{align*}
	\prrk \left( M_{\langle n \rangle}^{\otimes k} \right)^{\frac{1}{k}} \leq n^\rho + \prEps \, , \text{ so }
	\prrk \left( M_{\langle n^k \rangle} \right) = \prrk \left( M_{\langle n \rangle}^{\otimes k} \right)\leq (n^{\rho} + \prEps)^k \, .
	\end{align*}
Therefore, we have $\prrk \left( M_{\langle n^k \rangle} \right) \in \prO \left( n^{k \rho} \right)$ for $k \to \infty$, i.e., $\prrk \left( M_{\langle n^k \rangle} \right) \leq c' n^{k \rho}$ for some constant $c' > 0$ and $k \gg 0$.
Thus, if $n^{k-1} < m \leq n^k$ and $k \gg 0$ we obtain with $h \mapsto \prrk(M_{\langle h \rangle})$ being monotonically increasing (compare Remark~\ref{pr:remRkMnMonotone}) that
	\begin{align*}
	\prrk \left( M_{\langle m \rangle} \right) \leq \prrk \left( M_{\langle n^k \rangle} \right)
	\leq c' \big( n^{k} \big)^{\rho} < c' (nm)^{\rho} = c' n^\rho m^\rho \, .
	\end{align*}
Noting that $n^\rho$ is a constant as $n$ is fixed, we showed $\prrk(M_{\langle m \rangle}) \in \prO(m^\rho)$ for $m \to \infty$. Finally, equation~\eqref{pr:eqAsymptRankMn} implies $\omega \leq \rho$ as desired.
\end{proof}

Thus, the asymptotic rank of a $3$-tensor $f$ may be interpreted as a generalization of $\omega$, which captures the asymptotics of the family $M_{\langle n \rangle}$ of tensors in the sense of Theorem~\ref{pr:thmOmegaRk}.

Let us also mention that Strassen developed in the context of asymptotic rank the theory of asymptotic spectra and support functionals, see \cite{prStrAsymptSpectrum88} and \cite{prStrDegenerationAndComplexity91}. Thereby, he provides yet another way of expressing $\omega$, namely via the asymptotic spectrum. But these works of Strassen also have remarkable applications besides complexity theory, e.g., in quantum information theory. For the sake of the latter, Strassen's notions have been recently generalized (e.g., quantum functionals) in \cite{prCVZasymptoticSpectrum}, see also the detailed version \cite{prCVZarXivVersion}.

In the mentioned articles there appears the important notion of tight tensors. We introduce this together with conciseness for being able to formulate the upcoming Conjecture~\ref{pr:conjConciseTight}, see also Remark~\ref{rem:tight}.

\begin{defn}
Let $f \in \prC^{a} \otimes \prC^{b} \otimes \prC^{c}$ and let $I,J,K$ be finite sets of cardinality $a,b,c$ respectively.
	\begin{numlist}
	\item We say $f$ is \textbf{concise} if it has multilinear rank $(a,b,c)$, i.e., all contractions induced by $f$ are injective.
	
	\item A subset $S \subseteq I \times J \times K$ is said to be \textbf{tight}, if there exist injective functions $\alpha \colon I \to \prZ$, $\beta \colon J \to \prZ$ and $\gamma \colon K \to \prZ$ such that
		\begin{align*}
		\forall (i,j,k) \in S \colon \quad \alpha(i) + \beta(j) + \gamma(k) = 0 \, .
		\end{align*}
	\item The tensor $f$ is called \textbf{tight}, if there are bases $(a_i)_{i \in I}$, $(b_j)_{j \in J}$, $(c_k)_{k \in K}$ of $\prC^a$, $\prC^b$, $\prC^c$ respectively such that the support of $f$ with respect to the basis $(a_i \otimes b_j \otimes c_k)_{i,j,k}$ is a tight subset of $I \times J \times K$.
	\end{numlist}
\end{defn}

\begin{remark}\label{rem:tight}
It is a result by Strassen that a tensor in $\prC^{a} \otimes \prC^{b} \otimes \prC^{c}$ is tight if and only it is stabilized by 
a one parameter subgroup $\prC^{*}$ generated in some basis by three diagonal matrices,
each of them having distinct weights, see \cite[\S~2.1]{prCGLVWasymptRankConj}.
\end{remark}

\begin{ex}
The tensor $M_{\langle n \rangle}$ is concise and tight. To see this, remember that
	\begin{align*}
	M_{\langle n \rangle} = \sum_{i,j,k=1}^n E_{i,j} \otimes E_{j,k} \otimes E_{k,i} \, ,
	\end{align*}
where $E_{i,j}$ is the matrix with entry one at position $(i,j)$ and entry zero elsewhere. Moreover, set $[n] := \lbrace 1,\ldots,n \rbrace$ and let $(e_{i,j})_{i,j \in [n]}$ be the dual basis of $(E_{i,j})_{i,j \in [n]}$.

For conciseness, note that the contraction of $M_{\langle n \rangle}$ with respect to the first tensor factor is given by
	\begin{align*}
	\Gamma \colon \left( \prC^{n \times n} \right)^\vee \to \prC^{n \times n} \otimes \prC^{n \times n}, \quad e_{i,j} \mapsto \sum_{k=1}^n E_{j,k} \otimes E_{k,i} \, .
	\end{align*}
Since the system $(E_{j,k} \otimes E_{k,i})_{i,j,k \in [n]}$ is linearly independent, also the $\Gamma(e_{i,j})$ for $i,j \in [n]$ are seen to be linearly independent. Hence, $\Gamma$ is injective and the injectivity of the other two contractions follows similarly.

For tightness, note that the support of $M_{\langle n \rangle}$ is
	\begin{align*}
	\left\lbrace \big( (i,j), (j,k), (k,i) \big) \mid i,j,k \in [n] \right\rbrace \subseteq [n]^2 \times [n]^2 \times [n]^2
	\end{align*}
and consider the functions
	\begin{align*}
	\alpha \colon [n] \times [n] \to \prZ &, \quad (i,j) \mapsto i + jn \\
	\beta \colon [n] \times [n] \to \prZ &, \quad (j,k) \mapsto -jn + kn^2 \\
	\gamma \colon [n] \times [n] \to \prZ &, \quad (k,i) \mapsto -kn^2 - i \, .
	\end{align*}
These three functions are injective, e.g., $\alpha(i,j) = \alpha(k,l)$ yields $i=k$ by considering the unique $r \in [n]$ with $r \equiv \alpha(i,j) = \alpha(k,l) \mod n$, and then $i=k$ implies $j=l$. By construction, $\alpha(i,j) + \beta(j,k) + \gamma(k,i) = 0$ for all $i,j,k \in [n]$.\hfill\prExample
\end{ex}

\begin{conj}[{Strassen's Asymptotic Rank Conjecture, \cite[5.3]{prStrEuropeanCongress}}]\label{pr:conjConciseTight}\ \\
If $f \in \prC^m \otimes \prC^m \otimes \prC^m$ is concise and tight, then $\praR (f) = m$. %In particular, for $M_{\langle n \rangle} \in \prC^{n \times n} \otimes \prC^{n \times n} \otimes \prC^{n \times n}$ one obtains $\praR(M_{\langle n \rangle}) = n^2$ and hence $\omega  = 2$ by Theorem~\ref{pr:thmAsymptRankMn}.
\end{conj}

In \cite[Problem 15.5]{prAlgComplTheoryBook} it is asked, whether tightness is needed in Conjecture~\ref{pr:conjConciseTight}. Moreover, Strassen's Asymptotic Rank Conjecture and its variants are investigated in~\cite{prCGLVWasymptRankConj}.

Let us build a bridge to the conjecture $\omega = 2$. It is a consequence of Conjecture~\ref{pr:conjConciseTight}. Namely, for $M_{\langle n \rangle} \in \prC^{n \times n} \otimes \prC^{n \times n} \otimes \prC^{n \times n}$ Conjecture~\ref{pr:conjConciseTight} implies $\praR(M_{\langle n \rangle}) = n^2$ and hence $\omega  = 2$ by Theorem~\ref{pr:thmAsymptRankMn}.

\medskip

An interesting variant of the Definition~\ref{pr:defnAsymRk} of asymptotic rank is the following,  studied by Christandl, Jensen and Zuiddam in \cite{prCJZ18}, see also \cite{prCGJ19}.

\begin{defn}\label{pr:defnAsymRkCJZ}
For $f \in V_1 \otimes V_2 \otimes V_3$ the \textbf{tensor asymptotic rank} of $f$ is defined as
	\begin{align*}
	R^{\otimes}(f) := \lim_{k \to \infty} \big(  \prrk_{3k} \left( f^{\otimes k} \right) \big)^{\frac{1}{k}} \, ,
	\end{align*}
where  the rank is computed looking at $f^{\otimes k}$ as a $3k$-dimensional tensor.
\end{defn}
Of course, one has $\prrk_3(f^{\otimes k}) \leq \prrk_{3k}(f^{\otimes k})$ and therefore $\praR(f) \leq R^\otimes (f)$.

\medskip

In \cite{prCJZ18} the following interesting inequality is obtained.

\begin{thm}[{\cite[Corollary~12]{prCJZ18}}]
If $f \in V_1 \otimes V_2 \otimes V_3$, then $R^\otimes(f) \leq \prbrk(f)$.
\end{thm}

This leaves the interesting problem to compute $R^\otimes(M_{\langle n \rangle}) $.

%%%%%%% End: Asymptotic Rank %%%%%%%

%%%%%%% Begin: Exercises %%%%%%%
\newpage
\addcontentsline{toc}{section}{Exercises}
\section*{Exercises}\label{pr:secExercises}

These are the exercises from the fall school, including the newly added Exercises~5 and 6. We mention that Exercises~7 and 8 are intended as research projects, e.g., for the Thematic Einstein Semester.

In the following the tensor product is often suppressed when writing vectors.

\bigskip

\noindent{\bf Exercise 1.} A tensor $t\in\prC^a\otimes\prC^b\otimes\prC^c$
defines three contractions
	\begin{align*}
	(\prC^a)^\vee\to\prC^b\otimes\prC^c \, , \; (\prC^b)^\vee\to\prC^a\otimes\prC^c \, , \;
	(\prC^c)^\vee\to\prC^a\otimes\prC^b
	\end{align*}
and we call their ranks respectively $r_1(t)$, $r_2(t)$, $r_3(t)$. The multilinear rank of $t$ is the tuple $(r_1(t),r_2(t),r_3(t))$.

\begin{abclist}
\item Prove that $$\prrk(t)=1\dashrightarrow\left\{\begin{array}{c}
r_1(t)=1\\
r_2(t)=1\\
r_3(t)=1\end{array}\right.$$

\item Prove that there is some $t$ with $(r_1(t),r_2(t),r_3(t))=(1,2,2)$ and that
it is impossible to have $(r_1(t),r_2(t),r_3(t))=(1,1,2)$.

\item Prove for all $i$ that $r_i(t) \le \prrk(t)$ and then even ensure $r_i(t) \le \prbrk(t)$.

\item Prove $\prrk(t) \le r_j(t) r_k(t)$ for all $j, k$ with $j \neq k$. In particular, with part~c) we have $r_i(t) \leq r_j(t) r_k(t)$ for all $i,j,k$ with $\lbrace i,j,k \rbrace = \lbrace 1,2,3 \rbrace$.

\end{abclist}

\bigskip

\noindent{\bf Exercise 2.} (Geometric version)
We consider the natural action of the algebraic group $\prGL_2(\prC) \times \prGL_2(\prC) \times \prGL_2(\prC)$ on $\prC^2 \otimes \prC^2 \otimes \prC^2$ and the induced action on $\prP(\prC^2 \otimes \prC^2 \otimes \prC^2)$.
\begin{abclist}
\item Prove that, in the language of Exercise 1, the only admissible triples for $(r_1, r_2, r_3)$ are
$(1,1,1)$, $(1,2,2)$, $(2,1,2)$, $(2,2,1)$, $(2,2,2)$.

\item In the matrix space $\prP(\prC^2\otimes\prC^2)$ the variety of matrices of rank one is a smooth quadric surface $Q$.
Interpretate the three contractions $(\prC^2)^\vee\to\prC^2\otimes\prC^2$ as three pencils of matrices
$\prP^1\to\prP(\prC^2\otimes\prC^2)$ and characterize them depending on the intersection of the pencil with $Q$, according to each admissible triple. 
For example in the case $(1,1,1)$ the contraction maps $\prP^1\to\prP(\prC^2\otimes\prC^2)$ are degenerate: their images collapse to a single point contained in $Q$.
In the case $(1,2,2)$ there are two contraction maps whose image is a line all contained in $Q$. Go ahead with a complete description of all cases.

\item Prove that the cases $(1,1,1)$, $(1,2,2)$, $(2,1,2)$, $(2,2,1)$ correspond to a unique orbit, while
the case $(2,2,2)$ splits into two orbits. {\it Hint: the pencil may be transversal or tangent to $Q$.}

\item Describe the graph of the six orbit closures in $\prP(\prC^2\otimes\prC^2\otimes\prC^2)$ and the graph of the seven orbit closures
in $\prC^2\otimes\prC^2\otimes\prC^2$ (they include the zero orbit).
A reference is \cite[Example~14.4.5]{prGKZdiscriminants}.
\end{abclist}

\bigskip

\noindent{\bf Exercise 2'.} (This is the Algebraic Version of Exercise 2.)
Let $A$, $B$, $C$ be three vector spaces of dimension two,
with basis respectively given by $\{a_0,a_1\}$,
$\{b_0,b_1\}$,
$\{c_0,c_1\}$.
\begin{itemize}
\item[(1)] For any tensor $t=\sum_{i,j,k=0,1}t_{ijk}a_ib_jc_k$,
write it as a $2\times 2$ matrix with coefficients linear in $c_i$,
so as a {\it pencil} of $2\times 2$ matrices.

\item[(2)] Compute the condition that the previous matrix is singular,
it is a quadratic equation in $c_i$, whose roots correspond to a pair of points in $\prP^1$.

\item[(3)] Compute the condition that the previous pair of points consists of a double point,
get a polynomial of degree $4$ in $t_{ijk}$,
which is called the \textit{hyperdeterminant} of $t$ and we denote as $Det(t)$. It corresponds to a pencil tangent to $Q$,
in the geometric language of Exercise 2. (\textit{You may use a computer algebra system for computing $Det(t)$.})

\item[(4)] Show that the above pair gives the two summands
of $t$, when $t$ has rank two. Prove that in the dense orbit, over $\prC$, there is a unique decomposition
as a sum of two decomposable tensors (this was the main result by C. Segre).

\item[(5)]In the real case, prove that the sign of $Det(t)$
allows to detect if a real $2\times 2\times 2$ tensor has rank $2$ or $3$.

\item[(6)] Prove that $w=a_0b_0c_1+a_0b_1c_0+a_1b_0c_0$ has (complex) rank $3$. This is called a $W$-state in Quantum Information Theory.
Write infinitely many decompositions of $w$ as the sum of three decomposable tensors. {\it Hint: in the last summand you could modify with
$(a_0\sin\theta+a_1\cos\theta)(b_0\cos\theta+b_1\sin\theta)c_0$.}

\item[(7)] Prove that $a_0b_0c_0+a_0b_1c_1$ has infinitely many decompositions. Which are its multilinear ranks $(r_1, r_2, r_3)$? {\it Hint: in the first summand you could modify with
$a_0(b_0\cos\theta+b_1\sin\theta)(c_0\cos\theta+c_1\sin\theta)$.}

\end{itemize}

\bigskip

\noindent{\bf Exercise 3.} Let $\lbrace a_0,a_1 \rbrace$, $\lbrace b_0,b_1 \rbrace$ and $\lbrace c_0,c_1 \rbrace$ respectively be the standard basis of $\prR^2 \subseteq \prC^2$.

\begin{abclist}
\item The following $2\times 2\times 2$ tensors $t_1$, $t_2$, $t_3$ fill the first column of the following table. Which is which? Can you decompose them?

$$\begin{array}{c|c|c}t&\prrk_{\prR}(t)&\prrk_{\prC}(t)\\
\hline\textrm{?}&2&2\\
\hline\textrm{?}&3&2\\
\hline\textrm{?}&3&3
\end{array}$$

%$t_1=2 {a}_{0} {b}_{0} {c}_{0}+3 {a}_{0} {b}_{1} {c}_{0}-{a}_{1} {b}_{1}
 %5    {c}_{0}+6 {a}_{0} {b}_{0} {c}_{1}-4 {a}_{1} {b}_{0} {c}_{1}+11 {a}_{0}
   %  {b}_{1} {c}_{1}-9 {a}_{1} {b}_{1} {c}_{1}$
$t_1=4 {a}_{0} {b}_{0} {c}_{0}+2 {a}_{1} {b}_{0} {c}_{0}-{a}_{0} {b}_{1}
     {c}_{0}+2 {a}_{0} {b}_{0} {c}_{1}$

$t_2={a}_{0}{b}_{0} {c}_{0}+2 {a}_{1} {b}_{1} {c}_{0} +3 {a}_{1} {b}_{0} {c}_{1}
      +6 {a}_{0} {b}_{1} {c}_{1} $

$t_3={a}_{0}{b}_{0} {c}_{0}-2 {a}_{1} {b}_{1} {c}_{0} -3 {a}_{1} {b}_{0} {c}_{1}
      -6 {a}_{0} {b}_{1} {c}_{1} $

\item For the natural action of $\prGL_2(\prR)\times \prGL_2(\prR)\times \prGL_2(\prR)$ on $\prR^2\otimes\prR^2\otimes\prR^2$ write down the finite number of orbits.
\end{abclist}

\bigskip

\noindent{\bf Exercise 4.} Prove Lemma~\ref{pr:lemTensorProdOfMn}, i.e., that $M_{\langle n \rangle} \otimes M_{\langle m \rangle} = M_{\langle nm \rangle}$. To do so, use a certain identification $\prC^{n \times n} \otimes \prC^{m \times m} \cong \prC^{nm \times nm}$.

\bigskip

\noindent{\bf Exercise 5.} Show that the homogeneous polynomial $t = x^3 + 3x^2y \in \prC[x,y]$ is not the the sum of two cubes by showing that $t = (\alpha x + \beta y)^3 + (\gamma x + \delta y)^3$ induces a system of equations in $\alpha, \beta, \gamma, \delta$ with \emph{no} solutions in $\prC^4$.

\bigskip

\noindent{\bf Exercise 6.} Prove Fekete's Lemma, i.e., Lemma~\ref{pr:lemFekete}.

\bigskip

\noindent{\bf Exercise 7.} (More difficult) In \cite[Theorem~10.10.2.6]{prLandsbergBook} it is reported that the maximal rank for tensors in $\prC^2\otimes\prC^2\otimes\prC^2\otimes\prC^2$ is $4$. The proof by Brylinski examines the rank of the map $\prP^1\to\prP(\prC^2\otimes\prC^2\otimes\prC^2)$. Try to repeat the argument for tensors in $\prR^2\otimes\prR^2\otimes\prR^2\otimes\prR^2$
by using Exercise 3.

\bigskip

\noindent{\bf Exercise 8.} (More difficult) A rank in real tensor space is called typical if it is attained in a set with nonempty interior (equivalently, with positive volume).
It is known that the typical ranks of $\prR^2\otimes\prR^2\otimes\prR^2$ are $2$ and $3$.\\
What are the typical ranks for $\prR^2\otimes\prR^2\otimes\prR^2\otimes\prR^2$?

%%%%%%% End: Exercises %%%%%%%

\bibliographystyle{alpha}
\bibliography{references}

\end{document}